\documentclass[a4paper] {article}
[12pt]

\usepackage[top=2.5 cm, bottom=2.5 cm, left=3. cm, right=3. cm]{geometry}

\makeatletter
\renewcommand*\l@section{\@dottedtocline{1}{1.5em}{2.3em}}
\makeatother

\usepackage{amsfonts}
\usepackage{amssymb}
\usepackage[T1]{fontenc}

\usepackage{tikz}
\usetikzlibrary{calc}

\usepackage{CJK}
\usepackage{amsmath}

\usepackage{amsfonts}
\usepackage{amssymb}
\usepackage{amsthm}
\usepackage{amssymb}
\usepackage{enumerate}
\usepackage[calc]{picture}
\usepackage[all,cmtip]{xy}

\usepackage[mathscr]{eucal}
\usepackage{eqlist}

\usepackage{color}
\usepackage{abstract}
\usepackage[T1]{fontenc}

\theoremstyle{plain}
\newtheorem{theorem}{Theorem}
\newtheorem{proposition}[theorem]{Proposition}
\newtheorem{lemma}[theorem]{Lemma}

\newtheorem{example}[theorem]{Example}
\newtheorem{corollary}[theorem]{Corollary}

\theoremstyle{definition}
\newtheorem{definition}{Definition}

\usepackage{etoolbox}
\newtheoremstyle{myrem}
 {3pt}
 {3pt}
 {\normalsize}
 { }
 {\itshape}
 {:}
 { }
 {}

 \theoremstyle{myrem}
 \newtheorem{remark}{Remark}
 \appto\remark{\leftskip\parindent}
 \appto\remark{\rightskip\parindent}

\numberwithin{equation}{section}
\numberwithin{theorem}{section}

\begin{document}

\begin{center}

~~~

\bigskip

{\Large {\textbf {Differential Calculus  on  Hypergraphs  and   Mayer-Vietoris  Sequences  for  the   Constrained Persistent Homology }}}
 \vspace{0.58cm}\\

Shiquan Ren$^1$

\footnotetext[1]{The author  is    supported  by     China  Postdoctoral  Science  Foundation  2022M721023.     }

\smallskip

\begin{quote}
\begin{abstract}
 In  this  paper,  we  study  the  discrete  differential  calculus  on  hypergraphs   by  using   the  Kouzul  complexes.  
 We  define  the   constrained  (co)homology  for  hypergraphs  and  give  the  corresponding  Mayer-Vietoris  
 sequences.  We   prove  the  functoriality  of  the  Mayer-Vietoris  sequences  for    the   constrained   homology
  and  the   functoriality  of  the  Mayer-Vietoris  sequences  for    the   constrained   cohomology
 with respect to  morphisms  of  hypergraphs  induced  by  bijective maps  between  the  vertices.  
 Consequently,  we  obtain  the  Mayer-Vietoris  sequences  for  the  constrained  persistent  (co)homology 
 for    filtrations  of  hypergraphs.   
 As  applications,  we  propose   the  constrained  persistent   (co)homology  as  a  tool    for  the  computation  of  higher-dimensional  persistent homology  of  large networks.  
 \end{abstract}

\smallskip



\end{quote}

\end{center}

\medskip

\section{Introduction}

Let   $V$  be  a  vertex  set  with  a  total  order  $\prec$.   
A   {\it  hypergraph}  $\mathcal{H}$   with its  vertices  from   $V$  is  a  collection  of     non-empty  subsets   of  $V$   (cf.  \cite{berge}).   An  element  of  $\mathcal{H}$  is  called  a  {\it hyperedge}.  
  A  {\it  simplicial complex}  $\mathcal{K}$  with  its  vertices  from  $V$  
  is  a hypergraph such  that     any non-empty  subset   of  any  hyperedge  in  $\mathcal{K}$  is  still  a 
  hyperedge in  $\mathcal{K}$.   A   hyperedge  in   $\mathcal{K}$  is  also   called   a   {\it   simplex}.  
    An  {\it  independence  hypergraph}   $\mathcal{L}$  with  its  vertices  from  $V$  
  is  a hypergraph such  that    any  finite   superset (whose  elements are from  $V$)  of    any  hyperedge  in  $\mathcal{L}$   is  stll  a  hyperedge  in  $\mathcal{L}$.

Discrete  differential  calculus  has  been   initially  studied  by A. Dimakis  and    F. M\"{u}ller-Hoissen  \cite{d1,d2}.  
 Recently,  the  author  \cite{camb2023}   investigated  the  discrete  differential  calculus  for  the  differentials   $\frac{\partial}{\partial  v}$,  $v\in  V$,   on  simplicial  complexes as  well as   the  discrete  differential  calculus  
 for  the  differentials   $dv$,  $v\in  V$,  
on  independence  hypergraphs  and   constructed  the   constrained   homology  for  simplicial  complexes 
 as   well  as  the  constrained  cohomology  for  independence  hypergraphs.    
 The  functoriality  of  the  constrained  (co)homology   has  been    given    and  some  Mayer-Vietoris  sequences  for  the    constrained  (co)homology   have  been  proved   in \cite{mv}.

  Let $\mathcal{H}$  be  a  hypergraph.   The  associated  simplicial  complex  $\Delta\mathcal{H}$  is  the  
 smallest  simplicial  complex  such  that  each  hyperedge  of  $\mathcal{H}$   is  a  simplex  of  $\Delta\mathcal{H}$   (cf. \cite{hg1,parks}).  
 The  lower-associated  simplicial  complex  is  the  largest  simplicial  complex    $\delta\mathcal{H}$  such  
 that  each  simplex  of  $\delta\mathcal{H}$  is  a  hyperedge  of  $\mathcal{H}$  (cf.  \cite{jktr1000,jktr2,jktr3}).   
  The   associated   independence  hypergraph    $\bar\Delta\mathcal{H}$  
  is  the  smallest   independence  hypergraph      such  that   each  hyperedge  of  $\mathcal{H}$  
  is  a  hyperedge  of  $\bar\Delta\mathcal{H}$  (cf.  \cite{grh1}).  
  The   lower-associated   independence  hypergraph    $\bar\delta\mathcal{H}$  
  is  the  largest  independence  hypergraph      such  that   each  hyperedge  of  $\bar\delta\mathcal{H}$  
  is  a  hyperedge  of  $\mathcal{H}$  (cf.  \cite{grh1}).  
  With  the  help  of   $\Delta$,  $\delta$,  $\bar\Delta$  and  $\bar\delta$,  
  some  relations  among  random  hypergraphs   (cf.  \cite{f2022,jktr2}),  random  simplicial  complexes  (cf.  \cite{y2-2,cfh,f2022,9,y2-3,annals,jktr2})  and  random  independence 
  hypergraphs   can  be  obtained   (cf.   \cite[Theorem~1.1]{grh1}).

  In 1992,  A. D. Parks and S. L. Lipscomb  \cite{parks}   considered  the  simplicial  homology  of  $\Delta\mathcal{H}$     as  the  homology  of  $\mathcal{H}$  and  
   investigated  this  homology   to  study  the  acyclicity  of  $\mathcal{H}$.  
  In  2019,  inspired  by  the  path  homology  theory  of  digraphs    by  
  A.  Grigor'yan,  Y.  Lin, Y.  Muranov   and   S.-T. Yau  \cite{lin2,lin3,lin6},   S.  Bressan,  J,  Li,  S.  Ren and 
  J.  Wu  \cite{hg1}     constructed
  the   embedded  homology  of  hypergraphs  and   proved  a   (persistent)   Mayer-Vietoris   sequence  for  the  (persistent)  embedded  homology.     
  For  both  the    simplicial  homology  of  $\Delta\mathcal{H}$  and   the  embedded  homology  of  $\mathcal{H}$,  
  the  persistent  homology  of  low  dimensions  (for  example,  dimensions  $0$, $1$  and  $2$)   could   be  computed  even  if   the  number  of  vertices  in  $V$  as  well as   the  number  of   hyperedges  in $\mathcal{H}$  is  large.

  \smallskip

  In  this  paper,   we  study  the  discrete  differential  calculus  on  hypergraphs.  Let  $w$  be  a  weight  function on  $V$.    We  construct the   Kouzul  complex  for    the    differentials  $\frac{\partial}{\partial  v}$,  $v\in  V$,   as  well as  the  Kouzul  complex  for the  differentials  $dv$,  $v\in  V$,   with  respect  to  $w$.    
  We  define  the  $\mathcal{H}$-admissibility  for 
   $\frac{\partial}{\partial  v}$  as  well as    the  $\mathcal{H}$-admissibility  for  $dv$  in   Definition~\ref{def-22222}.  
  We  prove  in  Theorem~\ref{th-3.2ccv}  (main  result  I)  that  the  Kouzul  complex  for   $\frac{\partial }{\partial  v}$,  $v\in  V$ 
   (resp.  the  Kouzul  complex  for   $dv$,  $v\in  V$),     is  a long  exact sequence  if  there  are  finitely  many  vertices  $v$  such  that  $\frac{\partial}{\partial  v}$  is  $\mathcal{H}$-admissible (resp.   finitely  many  vertices  $v$  such  that    $ dv$  is  $\mathcal{H}$-admissible) and  $w$   is  non-vanishing  on  these  vertices.

   We  define  the  constrained  homology  for     $\mathcal{H}$  to  be  the  triple  consisting  of  the  constrained  homology  of  the  lower-associated  simplicial complex  $\delta\mathcal{H}$,  the  constrained  homology  of  the  associated  simplicial complex  $\Delta\mathcal{H}$,  and  the  homomorphism  from the  constrained  homology  of  $\delta\mathcal{H}$  to  the   constrained  homology  of  $\Delta\mathcal{H}$  induced  by  the  canonical  inclusion  from  $\delta\mathcal{H}$  into  $\Delta\mathcal{H}$.  
    We  define  the  constrained  cohomology  for     $\mathcal{H}$  to  be  the  triple  consisting  of  the  constrained  cohomology  of  the  lower-associated  independence  hypergraph  $\bar\delta\mathcal{H}$,  the  constrained  cohomology  of  the  associated  independence  hypergraph  $\Delta\mathcal{H}$,  and  the  homomorphism  from the  constrained  cohomology  of  $\bar\delta\mathcal{H}$  to  the   constrained  cohomology  of  $\bar\Delta\mathcal{H}$  induced  by  the  canonical  inclusion  from  $\bar\delta\mathcal{H}$  into  $\bar\Delta\mathcal{H}$.   
    In  Section~\ref{s4},  we  give  the   Mayer-Vietoris  sequences  for the  constrained  homology 
    of  hypergraphs  as  well  as  for   the   constrained  cohomology  of  hypergraphs.  
    In  Theorem~\ref{pr-5.829a}  (main  result  II),   we   prove  that  both  the    Mayer-Vietoris  sequences   for  the  constrained  homology 
    of  hypergraphs      
    and  the    Mayer-Vietoris  sequences   for  the  constrained  cohomology 
    of  hypergraphs  
    are  functorial  with respect  to  morphisms  of  hypergraphs  induced  by  bijective  maps  between  the  
    vertices.

     With  the  help  of    Theorem~\ref{pr-5.829a},  
      we  prove     the    Mayer-Vietoris  sequences    for  the constrained   persistent  homology    as  well as the    Mayer-Vietoris  sequences  for  the  constrained   persistent  cohomology 
       for  any  filtrations  of  hypergraphs   in  Theorem~\ref{th-091286}  (main  result   III).

  As  applications,  we   discuss  about  the  localizations  of persistent homology  for  large   networks   (for  example,  we    refer  to  \cite{physr,rev111,siam}  for  networks)  in  Section~\ref{s5}.  Persistent  homology  is  a  significant  computational  tool  in  topological  data  analysis  (for  example,  we    refer  to  \cite{pd2,pmd,pd1,superh,2005}  for  persistent  homology).  For  a  network  with  large  numbers  of   vertices  and  hyperedges,  
  the  $n$-th  persistent homology  is   computationally  expensive   if   $n$  is  large  as  well.   
   We  propose  the  constrained   persistent  (co)homology  as  a   possible  technique   for  the  computation  of  the  higher-dimensional  persistent  (co)homology  of     large  networks.    
  As  special  cases  of  the  constrained  
  persistent  (co)homology,  we      write the  usual  (co)boundary  
  map  (cf.  Example~\ref{ex-8.111}~(1)  and  (2))  as  a  sum  of  localized    (co)boundary  maps  such that    each  of  the   localized    (co)boundary  map   only  involves    small  numbers  of    vertices  and     hyperedges.   
  We  propose  the   localized  persistent  homology  of  the   (lower-)associated  simplicial  complex    for  a  sparse  network  and  propose    the   localized  persistent  cohomology  of  the   (lower-)associated  independence   hypergraph    for  a  dense  network.

  \smallskip

  The  paper  is  organized  as  follows.   In  Section~\ref{s2},  we  construct  the  Kouzul  complexes  
  for the  discrete  differential  calculus of $\frac{\partial}{\partial  v}$    and  $dv$   respectively  on a  discrete  set 
  $V$.   In  Section~\ref{s3},  we  study  the  discrete  differential   calculus  on  hypergraphs  and prove  
  Theorem~\ref{th-3.2ccv}.    In  Section~\ref{s4},    we  give  the  Mayer-Vietoris   sequences  for  
  the  constrained  (co)homology  of  hypergraphs  and  prove  the  functoriality  in  Theorem~\ref{pr-5.829a}. 
    In   Section~\ref{s88908},  we  construct  the    persistent   Mayer-Vietoris  sequences  for the  constrained  persistent  (co)homology  for   filtrations  of  hypergraphs  and  prove  the  functoriality  in  Theorem~\ref{th-091286}.    
  In  Section~\ref{s5},    we  propose  a  special  family  of  the  constrained  (co)homology,  which  is  called the   localized  persistent  cohomology,   for the  computations  of  higher-dimensional  persistent  homology  of  large networks.

\section{Discrete   Differential  Calculus     and  Kouzul  Complexes}\label{s2}

Let   $V$  be  a  discrete  set  whose elements are  called  {\it  vertices}.   
Let  $n\in\mathbb{N}$.    Let  $R$  be  a  commutative  ring  without   zero  divisors and   with  multiplicative  unit  $1$  such  that  
$2$  is  invertible.   
An {\it elementary $n$-path}     on $V$  is an  ordered sequence $v_0v_1\ldots v_n$ of  (not  necessarily  distinct)  $n+1$ vertices in  $V$   (cf.  \cite[Definition~2.1]{lin2},  \cite{lin1,lin3,lin4,lin6,lin5}).   
  A formal linear combination of elementary $n$-paths on $V$ with coefficients in  $R$  is called an {\it $n$-path}  on $V$.
  Denote by $\Lambda_n(V)$  the  free  $R$-module  of all the  $n$-paths  on  $V$ (cf.  \cite[Subsection~2.1]{lin2},   \cite{lin1,lin3,lin4,lin6,lin5}).     We have a graded  free  $R$-module   
$
\Lambda_*(V)=\bigoplus_{n=0}^\infty \Lambda_n(V)$.  
For any $v\in V$,   the {\it partial derivative} on  $\Lambda_*(V)$ with respect to $v$ is defined  to be  a sequence of   $R$-linear maps   (cf.  \cite[Subsection~3.2]{camb2023},   \cite{sid})
\begin{eqnarray}\label{eq-1.x1}
\frac{\partial}{\partial v}: ~~~ \Lambda_n(V)\longrightarrow \Lambda_{n-1}(V),~~~
n\in \mathbb{N}
\end{eqnarray}
given by
\begin{eqnarray}\label{eq-1.1}
\frac{\partial}{\partial v} (v_0v_1 \ldots v_n)=\sum_{i=0}^n (-1)^i  \delta(v,v_i) v_0\ldots \widehat{v_i}\ldots v_n,  
\end{eqnarray}
where      we  use the notation  $\delta (v,v_i)=1$  if $v=v_i$  and  $\delta(v,v_i)=0$  if  $v\neq v_i$,   
and   the {\it partial differentiation}  $dv$  on   $\Lambda_*(V)$  with respect to $v$  is  defined  to be a sequence of  $R$-linear maps
 (cf.  \cite[Subsection~3.3]{camb2023},   \cite{sid}) 
\begin{eqnarray}\label{eq-1.x2}
d v:  ~~~  \Lambda_{n}(V)\longrightarrow \Lambda_{n+1}(V), ~~~n\in \mathbb{N}
\end{eqnarray}
given  by
\begin{eqnarray} \label{eq-2.x5}
dv (u_0u_1\ldots u_{n-1})=\sum_{i=0}^n (-1)^i u_0u_1\ldots u_{i-1} v u_i u_{i+1}\ldots u_{n-1}.
\end{eqnarray}
Let  ${\rm  Ext}_*(V)$  be  the  exterior  algebra    generated  by  $\frac{\partial}{\partial  v}$,   $v\in  V$,   over  $R$
 and  let   ${\rm  Ext}^*(V)$  be  the  exterior  algebra  generated  by  $ dv $,   $v\in  V$,   over  $R$.  
It  is  proved  in  \cite[Lemma~3.1]{camb2023}  that  
$\frac{\partial}{\partial  v}\circ  \frac{\partial}{\partial  u}=-\frac{\partial}{\partial  u}\circ  \frac{\partial}{\partial  v}$
  and   in  \cite[Lemma~3.3]{camb2023}  that  
$dv \circ du=-du\circ  dv$  for  any  $u,v\in  V$. 
 Thus  both ${\rm  Ext}_*(V)$  and  ${\rm  Ext}^*(V)$   act  on  $\Lambda_*(V)$  such that the 
 compositions  of  maps are  represented  by  
 exterior  products.    
  Let   $w:  V\longrightarrow  R$  be  any  $R$-valued  function  on  $V$.   
 We  define  
 \begin{eqnarray*}
 \delta_n(w): ~~~ {\rm  Ext}_{n }(V)\longrightarrow     {\rm  Ext}_{n-1 }(V) 
 \end{eqnarray*}
 to  be  an  $R$-linear  map  given   by 
 \begin{eqnarray*}
 \delta_n(w)(\frac{\partial }{\partial  v_1}\wedge \cdots\wedge \frac{\partial }{\partial  v_n} )= \sum_{i=1}^n  (-1)^i  
 w(v_i)  \frac{\partial }{\partial  v_1}\wedge \cdots\wedge\widehat{\frac{\partial }{\partial  v_i}}\wedge\cdots \wedge \frac{\partial }{\partial  v_n} 
 \end{eqnarray*}
 and   define  
 \begin{eqnarray*}
 \delta^n(w): ~~~ {\rm  Ext}^{n }(V)\longrightarrow   {\rm  Ext}^{n-1 }(V) 
 \end{eqnarray*}
 to  be  an  $R$-linear  map  given   by 
 \begin{eqnarray*}
 \delta^n(w)(d v_1 \wedge \cdots\wedge   dv_n)= \sum_{i=1}^n  (-1)^i  
 w(v_i)    d v_1 \wedge \cdots\wedge\widehat{d v_i }\wedge\cdots \wedge  d  v_n.   
 \end{eqnarray*}
 
 \begin{lemma}\label{le-abc}
Let  $w:  V\longrightarrow  R$.  Let  $m,n\in \mathbb{N}$.   
\begin{enumerate}[(1).]
\item
For  any   $\xi_1 \in  {\rm  Ext}_n(V)$  and  any  $\xi_2\in  {\rm  Ext}_m(V)$,   we  have 
\begin{eqnarray*}
\delta_{n+m}(w)(\xi_1\wedge \xi_2)  = \delta_n(\xi_1)  \wedge  \xi_2  + (-1)^n \xi_1\wedge \delta_m(\xi_2);  
\end{eqnarray*}
\item
For  any   $\omega_1 \in  {\rm  Ext}^n(V)$  and  any  $\omega_2\in  {\rm  Ext}^m(V)$,   we  have 
\begin{eqnarray*}
\delta^{ n+m}(w)(\omega_1\wedge \omega_2)  = \delta^n(\omega_1)  \wedge  \omega_2  + (-1)^n \omega_1\wedge \delta^m(\omega_2).  
\end{eqnarray*}
\end{enumerate}
 \end{lemma}
 
 \begin{proof}
 We  only  prove  (1)  for   $\xi_1=    \frac{\partial }{\partial  v_{1}}\wedge \cdots\wedge \frac{\partial }{\partial  v_{n}} $  and   $\xi_2=   \frac{\partial }{\partial  u_{ 1}}\wedge \cdots\wedge \frac{\partial }{\partial  u_{m}} $.      \begin{eqnarray*}
\delta_{n+m}(w)(\xi_1\wedge \xi_2)  &=& \sum_{i=1}^n  (-1)^i  
 w(v_i)  \frac{\partial }{\partial  v_1}\wedge \cdots\wedge\widehat{\frac{\partial }{\partial  v_i}}\wedge\cdots \wedge \frac{\partial }{\partial  v_n} \wedge  \frac{\partial }{\partial  u_1}\wedge \cdots\wedge \frac{\partial }{\partial  u_m}
 \\
&& +\sum_{j=1}^{m}  (-1)^{n+j}  
 w(u_j)  \frac{\partial }{\partial  v_1}\wedge \cdots\wedge  \frac{\partial }{\partial  v_n} \wedge  \frac{\partial }{\partial  u_1}\wedge\cdots\wedge \widehat{\frac{\partial }{\partial  u_j}}\wedge\cdots \wedge  \frac{\partial }{\partial  u_m}
 \\
 &=& \delta_n(\xi_1)  \wedge  \xi_2  + (-1)^n \xi_1\wedge \delta_m(\xi_2).  
 \end{eqnarray*}
 We  obtain  (1).  
 The  proof  of  (2)  is  analogous.  
 \end{proof}

\begin{proposition}\label{pr-88.8}
For    any    $w:  V\longrightarrow  R$    we  have  two  chain  complexes
\begin{eqnarray*}
&\xymatrix{
\cdots \ar[r]^-{\delta_{n+1}(w)} &{\rm  Ext}_{n }(V)\ar[r]^-{\delta_n(w)}  & {\rm  Ext}_{n-1 }(V)
\ar[r]^-{\delta_{n-1}(w)} & \cdots \ar[r]^-{\delta_{2}(w)} &{\rm  Ext}_{1 }(V) \ar[r]^-{\delta_{1}(w)}  & R\ar[r]^-{\delta_{0}(w)}   &0,   
}\\
&\xymatrix{
\cdots \ar[r]^-{\delta^{n+1}(w)} &{\rm  Ext}^{n }(V)\ar[r]^-{\delta^n(w)}  & {\rm  Ext}^{n-1 }(V)
\ar[r]^-{\delta^{n-1}(w)} & \cdots  \ar[r]^-{\delta^{2}(w)} &{\rm  Ext}^{1 }(V) \ar[r]^-{\delta^{1}(w)}  & R\ar[r]^-{\delta^{0}(w)}  &0. 
}
\end{eqnarray*}
Moreover,  if   $V$  is  a  finite  set  and   $w$  is  non-vanishing,    then  both     chain complexes  are   long  exact  sequences.  
\end{proposition}

\begin{proof}
We  prove  the  statements  for  the  first  sequence.  The  proof for the second  sequence  is  analogous.  
     Let 
 $\frac{\partial }{\partial  v_1}\wedge \cdots\wedge \frac{\partial }{\partial  v_n}\in  {\rm  Ext}_n(V)$.  
By  definition, 
\begin{eqnarray*}
&&\delta_{n-1}(w) \delta_{n}(w)(\frac{\partial }{\partial  v_1}\wedge \cdots\wedge \frac{\partial }{\partial  v_n})\\
&=&
\delta_{n-1}(w)\Big(\sum_{i=1}^n  (-1)^i  
 w(v_i)  \frac{\partial }{\partial  v_1}\wedge \cdots\wedge\widehat{\frac{\partial }{\partial  v_i}}\wedge\cdots \wedge \frac{\partial }{\partial  v_n}\Big)\\
 &=& \sum_{i,j=1\atop  j<i}^n   (-1)^{i+j}  
 w(v_i)   w(v_j)  \frac{\partial }{\partial  v_1}\wedge \cdots\wedge\widehat{\frac{\partial }{\partial  v_j}}\wedge\cdots\wedge \widehat{\frac{\partial }{\partial  v_i}}\wedge\cdots \wedge \frac{\partial }{\partial  v_n}\\
&&  +  \sum_{i,j=1\atop  j>i}^n   (-1)^{i+j-1}  
 w(v_i)   w(v_j)  \frac{\partial }{\partial  v_1}\wedge \cdots\wedge\widehat{\frac{\partial }{\partial  v_i}}\wedge\cdots\wedge \widehat{\frac{\partial }{\partial  v_j}}\wedge\cdots \wedge \frac{\partial }{\partial  v_n}\\
 &=& 0.  
\end{eqnarray*}
The  last   equality  follows  from  the  commutativity  of   $R$.  
Thus   $\delta_{n-1}(w) \delta_{n}(w)=0$,  which  implies that     the  first  sequence  is  a  chain  complex  and    
\begin{eqnarray}\label{eq-2.cxz1}
{\rm  Im}\delta_{n}(w)  \subseteq  {\rm  Ker}  \delta_{n-1}(w).
\end{eqnarray}  
Suppose  in  addition  that  $V$  is   a  finite  set  and  $w:  V\longrightarrow  R\setminus \{0\}$ is  non-vanishing.  Let  $\xi\in {\rm  Ext}_{n-1 }(V) $  such that  $\delta_{n-1}(w)  (\xi)=0$.   We  assert  that  there exists  $\eta\in   {\rm  Ext}_{n  }(V) $  such  that  
$\delta_n(w)(\eta)=\xi$.  
Without loss  of  generality,  we  assume  $\xi\neq  0$.  
 We prove   this  assertion  by  induction on  the  cardinality  $|V|$  of  $V$.

{\it  Step~1}.  Suppose $|V|=1$.   Let  $v$  be  the  unique  element  in  $V$.   Then  
\begin{eqnarray*}
{\rm  Ext}_n(V)=
\begin{cases}
0,  &n\geq  2,\\
R(\frac{\partial}{\partial  v}),   & n=1,\\
R,  &n=0.
\end{cases} 
\end{eqnarray*}
For  any  $\xi\in {\rm  Ext}_{1 }(V) $,  write  $\xi=r\frac{\partial}{\partial  v}$  for some  $r\in  R$.  
Then  $\delta_{1}(w)  (\xi) = r w(v)$.   Since  $R$  is  assumed  to  have  no  zero  divisors  and  
$w$  is  assumed  non-vanishing,  it  follows  that  $\delta_{1}(w)  (\xi)=0$  iff  $r=0$  iff  $\xi=0$.  
The  assertion  holds.

 {\it  Step~2}.  Suppose  the  assertion  holds  for  $|V|=m-1$.   We  are  going to  prove  the  assertion  
 for  $|V|=m$.   Write  
 \begin{eqnarray}\label{eq-1133}
 \xi=\xi_1\wedge  \frac{\partial}{\partial  v_m} +\xi_2 
 \end{eqnarray}
 where  $\xi_1\in  {\rm  Ext}_{n-1}(V)$  and  $\xi_2\in  {\rm  Ext}_{n}(V)$  such  that 
  both  $\xi_1$  and  $\xi_2$  do  not  contain  $\frac{\partial}{\partial  v_m}$, i.e.   $\xi_1\in  {\rm  Ext}_{n-1}(V\setminus\{v_m\})$  and  $\xi_2\in  {\rm  Ext}_{n}(V\setminus\{v_m\})$.    
  With  the  help  of   Lemma~\ref{le-abc}~(1),   we  have 
  \begin{eqnarray}
  0&=& \delta_{n}(w)  (\xi_1\wedge  \frac{\partial}{\partial  v_m} +\xi_2) \nonumber\\
  &= & \big(\delta'_{n-1}(w) (\xi_1)\big) \wedge     \frac{\partial}{\partial  v_m} + (-1)^{n-1} \xi_1\wedge \delta_{1} (w)(\frac{\partial}{\partial  v_m})   +    \delta'_{n}(w) (\xi_2)\nonumber\\
  &=&\big(\delta'_{n-1}(w) (\xi_1)\big) \wedge     \frac{\partial}{\partial  v_m} + (-1)^{n-1} w(v_m) \xi_1    +    \delta'_{n}(w) (\xi_2),
  \label{eq-5igk1}
  \end{eqnarray}
  where  $ \delta'_{*}(w) $  denotes   the   boundary  maps  of  the  chain  complex  for  $|V|=m-1$.  
  Since  both  $\delta'_{n-1}(w) (\xi_1)$  and  $ (-1)^{n-1} w(v_m) \xi_1   +    \delta'_{n}(w) (\xi_2)$  do  not  contain  $\frac{\partial}{\partial  v_m}$,  we  have  
  \begin{eqnarray}
  &\delta'_{n-1}(w) (\xi_1)=0, \label{eq-0012}\\
  & (-1)^{n-1}  w(v_m) \xi_1    +    \delta'_{n}(w) (\xi_2)
  =0.  \label{eq-0022}
  \end{eqnarray}
  By  induction  and  (\ref{eq-0012}),  there  exists  $\eta_1\in   {\rm  Ext}_{n}(V\setminus\{v_m\})$   such  that  
  \begin{eqnarray}\label{eq-3355}
  \delta'_{n}(w) (\eta_1)=\xi_1.   
  \end{eqnarray}
   Thus  with the help of   (\ref{eq-0022}),  
  \begin{eqnarray*}
  0&=&(-1)^{n-1}  w(v_m)\delta'_{n}(w) (\eta_1)    +    \delta'_{n}(w) (\xi_2)\\
  &=& \delta'_{n}(w)  ( (-1)^{n-1}  w(v_m)  \eta_1  +  \xi_2). 
  \end{eqnarray*}
  By  induction,  there  exists   $\eta_2\in  {\rm  Ext}_{n+1}(V\setminus\{v_m\}) $  such  that  
  \begin{eqnarray}\label{eq-2233}
 (-1)^{n-1}  w(v_m)  \eta_1  +  \xi_2 =  \delta'_{n+1}(w)  (\eta_2).  
  \end{eqnarray}  
  Therefore, by  Lemma~\ref{le-abc}~(1),   (\ref{eq-1133}),  (\ref{eq-3355})  and  (\ref{eq-2233}),  
  \begin{eqnarray*}
  \xi     &=& \delta'_{n}(w) (\eta_1) \wedge  \frac{\partial}{\partial  v_m}  + (-1)^{n }  w(v_m)  \eta_1  +  \delta'_{n+1}(w)  (\eta_2)\\
  &=&  \delta_{n+1}(w) (\eta_1 \wedge    \frac{\partial}{\partial  v_m} +  \eta_2).  
  \end{eqnarray*}
  Let  $\eta= \eta_1 \wedge    \frac{\partial}{\partial  v_m} +  \eta_2$.  Then  $ \delta_{n+1}(w)(\eta) =\xi$.  The  assertion  follows  for     $|V|=m$.

  Summarizing  Step~1  and  Step~2,  by  induction on $|V|$,  we  obtain  the  assertion.   It  follows  that  
  \begin{eqnarray}\label{eq-2.cxz2}
{\rm  Im}\delta_{n+1}(w)  \supseteq  {\rm  Ker}  \delta_{n }(w).
\end{eqnarray}  
By  (\ref{eq-2.cxz1})  and  (\ref{eq-2.cxz2}),  the  first chain  complex  is  a  long  exact  sequence.  
\end{proof}

\begin{remark}
Proposition~\ref{pr-88.8}   is  a  discrete  analog  of  \cite[Proposition~3.18]{siu}. 
\end{remark}

\begin{definition}
We  call  the  chain  complexes  in  Proposition~\ref{pr-88.8}  the  {\it   Kouzul   complexes}  with  respect to  $w$   and denote them  by  $K_*(V,w;R)$  and  $K^*(V,w;R)$  respectively.  
\end{definition}

Let  $\varphi:  V\longrightarrow  V $  be  a  bijection.   We  have  an  induced  isomorphism  of  exterior  algebras 
 ${\rm  Ext}_*(\varphi):  {\rm  Ext}_*(V)\longrightarrow  {\rm  Ext}_*(V)$  sending   $\frac{\partial}{\partial  v}$  to  
 $\frac{\partial}{\partial  \varphi(v)}$  as  well  as  an induced  isomorphism  of  exterior  algebras  ${\rm  Ext}^*(\varphi):  {\rm  Ext}^*(V)\longrightarrow  {\rm  Ext}^*(V)$  sending   $d  v $  to  
 $ d  \varphi(v) $.  
 
 \begin{proposition}\label{pr-068a}
 For  any  bijection  $\varphi:  V\longrightarrow  V$  and  any  $w:  V\longrightarrow  R$  such  that  there  exists  an  endomorphism  of  rings  $ f_{\varphi,w}:  R\longrightarrow  R$   sending   $w(v)$  to  
 $w(\varphi(v))$  for  any  $v\in  V$,  we  have     induced    chain  maps  
 \begin{eqnarray}
 {\rm  Ext}_*(\varphi): &&~  K_*(V,w;R)\longrightarrow  K_*(V,w;R), 
 \label{eq-abcdxy1}\\
 {\rm  Ext}^*(\varphi): &&~  K^*(V,w;R)\longrightarrow  K^*(V,w;R).   
  \label{eq-abcdxy2}
 \end{eqnarray}
 Moreover,   if   $V$  is  a  finite  set  and   $w$  is  non-vanishing,  then  both  (\ref{eq-abcdxy1})
  and   
  (\ref{eq-abcdxy2})  are   morphisms  of  long  exact  sequences.  
 \end{proposition}
 
 \begin{proof}
 Let  $\varphi:  V\longrightarrow  V $  be  a  bijection.   Let  $w:  V\longrightarrow  R$  be  a  map  such  that  there  exists  an  endomorphism  of  rings  $  f_{\varphi,w}:  R\longrightarrow  R$   sending   $w(v)$  to  
 $w(\varphi(v))$  for  any  $v\in  V$.   We   have  a  commutative  diagram 
 \begin{eqnarray*}
 \xymatrix{
\cdots \ar[r]^-{\delta_{n+1}(w)} &{\rm  Ext}_{n }(V)\ar[r]^-{\delta_n(w)} \ar[d]_{{\rm  Ext}_n(\varphi)} & {\rm  Ext}_{n-1 }(V)
\ar[r]^-{\delta_{n-1}(w)} \ar[d]_{{\rm  Ext}_{n-1}(\varphi)} & \cdots \ar[r]^-{\delta_{2}(w)} &{\rm  Ext}_{1 }(V) \ar[d]_{{\rm  Ext}_1(\varphi)} \ar[r]^-{\delta_{1}(w)}  & R\ar[r]^-{\delta_{0}(w)}\ar[d]_{ f_{\varphi,w}}   &0\\
\cdots \ar[r]^-{\delta_{n+1}(w)} &{\rm  Ext}_{n }(V)\ar[r]^-{\delta_n(w)}  & {\rm  Ext}_{n-1 }(V)
\ar[r]^-{\delta_{n-1}(w)} & \cdots \ar[r]^-{\delta_{2}(w)} &{\rm  Ext}_{1 }(V) \ar[r]^-{\delta_{1}(w)}  & R\ar[r]^-{\delta_{0}(w)}   &0      
}
 \end{eqnarray*}
 and  a  commutative  diagram
  \begin{eqnarray*}
 \xymatrix{
\cdots \ar[r]^-{\delta^{n+1}(w)} &{\rm  Ext}^{n }(V)\ar[r]^-{\delta^n(w)} \ar[d]_{{\rm  Ext}^n(\varphi)} & {\rm  Ext}^{n-1 }(V)
\ar[r]^-{\delta^{n-1}(w)} \ar[d]_{{\rm  Ext}^{n-1}(\varphi)} & \cdots \ar[r]^-{\delta^{2}(w)} &{\rm  Ext}^{1 }(V) \ar[d]_{{\rm  Ext}^1(\varphi)} \ar[r]^-{\delta^{1}(w)}  & R\ar[r]^-{\delta^{0}(w)}\ar[d]_{ f_{\varphi,w}}   &0\\
\cdots \ar[r]^-{\delta^{n+1}(w)} &{\rm  Ext}^{n }(V)\ar[r]^-{\delta^n(w)}  & {\rm  Ext}^{n-1 }(V)
\ar[r]^-{\delta^{n-1}(w)} & \cdots \ar[r]^-{\delta^{2}(w)} &{\rm  Ext}^{1 }(V) \ar[r]^-{\delta^{1}(w)}  & R\ar[r]^-{\delta^{0}(w)}   &0,         
}
 \end{eqnarray*}
 which give   the   chain  maps    (\ref{eq-abcdxy1})
  and   
  (\ref{eq-abcdxy2})   respectively.  
 Moreover,    if   $V$  is  a  finite  set  and   $w$  is  non-vanishing,  then  all  the  rows  in  the  last two  commutative  diagrams  are  long  exact  sequences.   Thus   (\ref{eq-abcdxy1})
  and   
  (\ref{eq-abcdxy2})   are   morphisms  of  long  exact  sequences.  
 \end{proof}

\section{Discrete  Differential  Calculus  on    Hypergraphs}\label{s3}

 An   elementary  $n$-path $v_0 v_1\ldots v_n$  on $V$  is  called   {\it  non-simplicial}    if there exist integers  $0\leq i<j\leq n$  such that  either  $v_j\prec v_i$  or  $v_j=v_i$  (cf.  \cite[Definition~4.1]{camb2023}).
 Let $\mathcal{O}_n(V)$  be  free   $R$-module  generated  by all the  non-simplicial   elementary  $n$-paths   on $V$.    Let  $\tilde  \Lambda_n(V)$  be  the   sub-$R$-module   of  $\Lambda_n(V)$  whose  canonical  generators  
  are  all  the  elementary  $n$-paths  $v_0v_1\ldots  v_n$   such  that   $v_0,v_1,\ldots,  v_n$  are  distinct  
  and  $v_0\prec   v_1\prec\cdots\prec   v_n$.   
    Let  $\tilde \Lambda_*(V)=\bigoplus _{n=0}^\infty   \tilde \Lambda_n(V)$  and    $\mathcal{O}_*(V)=\bigoplus _{n=0}^\infty  \mathcal{O}_n(V)$.  Then
     $
    \Lambda_*(V)= \tilde  \Lambda_*(V)  \oplus \mathcal{O}_*(V)$   is  a  direct  sum  of  graded  $R$-modules. 
    The  restriction  of  (\ref{eq-1.x1})   gives  a  homomorphism  of  graded  $R$-modules   
  \begin{eqnarray}\label{eq-nnv1}
 {\frac{\partial}{\partial v}}: ~~~ \tilde \Lambda_n(V)\longrightarrow \tilde \Lambda_{n-1}(V),~~~
n\in \mathbb{N}.  
\end{eqnarray}
By composing with the  canonical  quotient  map  $\pi:  \Lambda_*(V)\longrightarrow \tilde \Lambda_*(V)$  sending   $\mathcal{O}_*(V)$  to zero  and  sending  $\tilde \Lambda_*(V)$  identically to itself,   (\ref{eq-1.x2})  gives  a graded    map  of  $R$-modules  
 \begin{eqnarray}\label{eq-nnv2}
  {d v}:  ~~~  \tilde \Lambda_{n}(V)\overset{dv}{\longrightarrow}  \Lambda_{n+1}(V)\overset{\pi}{\longrightarrow}\tilde  \Lambda_{n+1}(V), ~~~n\in \mathbb{N}.  
\end{eqnarray}

 Let  $\Delta[V]$  be  the  collection of all   the  non-empty  finite  subsets  of  $V$.  
   An  element  $\sigma\in \Delta[V]$  consisting  of  $n+1$  vertices in  $V$   is  an  {\it  $n$-hyperedge}  on  $V$,  
 denoted as  $\sigma^{(n)}$  or  $|\sigma|=n+1$.    
   Without  loss  of  generality,  we  choose  the  representative    $  v_0v_1\ldots   v_n$  of  $\sigma^{(n)}$  such that    $v_0\prec  v_1\prec \cdots\prec v_n$  and  write 
 $\sigma^{(n)}=  \{v_0,v_1,\ldots,v_n\}$.   
Let
$
R_n(\Delta[V] )$   be  the  free  $R$-module   
  generated  by  all  the $n$-hyperedges on $V$.   Take   the  direct  sum
$R_*(\Delta[V] )=\bigoplus_{n= 0}^\infty  R_n(\Delta[V] )$.  
Choosing    the elementary  $n$-path  $v_0v_1\ldots  v_n$  on  $V$ satisfying 
 $v_0\prec  v_1\prec \cdots  \prec  v_n$  as   an  representative of $\sigma^{(n)}$,  
  we  have  that    $R_*(\Delta[V] )$   can  be  identified  with  
  $\tilde  \Lambda_*(V)$.   We   rewrite   (\ref{eq-nnv1})   and    (\ref{eq-nnv2})      respectively  as       homomorphisms  of  graded  $R$-modules   
   \begin{eqnarray}
    \frac{\partial}{\partial  v}: && ~~~ R_{n }(\Delta[V] )\longrightarrow  R_{n-1}(\Delta[V] ), 
    \label{eq-mmm1}\\
           dv: && ~~~ R_n(\Delta[V] )\longrightarrow    R_{n+1}(\Delta[V] )     
           \label{eq-mmm2}
\end{eqnarray}
   for $n\in \mathbb{N}$.  

A   {\it  hypergraph}    with its  vertices  from   $V$  is  an  arbitrary  subset   $\mathcal{H}\subseteq \Delta[V]$.  
  A  {\it  simplicial complex}  $\mathcal{K}$  with  its  vertices  from  $V$  
  is  a hypergraph such  that  for  any  $\sigma\in \mathcal{K}$  and  any non-empty  subset 
  $\tau\subseteq \sigma$,  it  holds  $\tau\in \mathcal{K}$.      A  hyperedge  in  a  simplicial  complex  is  called  a  {\it  simplex}.   An  {\it  independence  hypergraph}   $\mathcal{L}$  with  its  vertices  from  $V$  
  is  a hypergraph such  that  for  any  $\sigma\in \mathcal{L}$  and  any  finite   superset 
  $\tau\supseteq \sigma$,  it  holds  $\tau\in \mathcal{L}$.   
Let  $\mathcal{H}$  be  a  hypergraph  with  its  vertices  from  $V$.  
Let  $R_n(\mathcal{H} )$  be the  free   $R$-module  generated  by all  the $n$-hyperedges  in $\mathcal{H}$.
Take  the  direct  sum $R_*(\mathcal{H})=\bigoplus_{n=  0}^\infty  R_n(\mathcal{H})$.  
 Then   $R_*(\mathcal{H})$   is  a  graded  sub-$R$-module  of  $R_*(\Delta[V] ) $.  
 
 \begin{lemma}\label{le-3.1ccv}
 For  any  hypergraph   $\mathcal{H}$   with  its  vertices  from  $V$,    
 \begin{enumerate}[(1).]
 \item
 $\mathcal{H}$  is a  simplicial  complex  iff   (\ref{eq-mmm1})  induces  a  homomorphism  
 \begin{eqnarray}\label{eq-0086a}
 \frac{\partial}{\partial  v}:~~~   R_{n }(\mathcal{H})\longrightarrow   R_{n-1}(\mathcal{H})
 \end{eqnarray}
   for  any  $v\in  V$  and  any  $n\in \mathbb{N}$; 
 \item
 $\mathcal{H}$  is an   independence  hypergraph  iff   (\ref{eq-mmm2})  induces  a  homomorphism  
 \begin{eqnarray}\label{eq-0086b}
 dv:  ~~~ R_{n }(\mathcal{H})\longrightarrow   R_{n+1}(\mathcal{H})
 \end{eqnarray}
   for  any  $v\in  V$  and  any  $n\in \mathbb{N}$. 
 \end{enumerate}
 \end{lemma}
 \begin{proof}
 (1).  Suppose  $\mathcal{H}$  is  a simplicial  complex.  Then  for  any    simplex  $\sigma\in \mathcal{H}$   and  any   $v\in  \sigma$,  we  have  that     $\sigma\setminus\{v\}$  is  either the  empty-set  or  a  simplex  in  $\mathcal{H}$.    Since 
 \begin{eqnarray*}
 \frac{\partial}{\partial  v} (\sigma)=
 \begin{cases}
 0,   & {\rm~if~} v\notin \sigma  {\rm ~or~ } \dim\sigma=0,\\
  \pm  (\sigma\setminus\{v\}),    &  {\rm~if~}  v\in \sigma {\rm~and~}  \dim\sigma\geq 1, 
 \end{cases}
 \end{eqnarray*}
 we  have    the  induced  homomorphism   $\frac{\partial}{\partial  v}:   R_{n }(\mathcal{H})\longrightarrow   R_{n-1}(\mathcal{H})$  for  any  $v\in  V$  and  any  $n\in \mathbb{N}$.

 Conversely,  suppose  we  have  an  induced   homomorphism   $\frac{\partial}{\partial  v}:   R_{n }(\mathcal{H})\longrightarrow   R_{n-1}(\mathcal{H})$  for  any  $v\in  V$  and  any  $n\in \mathbb{N}$.  
 Then  for  any  $\sigma\in \mathcal{H}$  such  that  $\dim\sigma\geq 1$,  we  have that  $\sigma\setminus\{v\}$ 
 is  a  hyperedge  in  $\mathcal{H}$    for  any  $v\in \sigma$.  By  induction,  for  any  non-empty  subset  $\tau\subseteq\sigma$,  we have $\tau\in \mathcal{H}$.    Therefore,  $\mathcal{H}$  is  a  simplicial complex.

 (2).   Suppose  $\mathcal{H}$  is  an  independence  hypergraph.   Then  for  any    hyperedge   $\sigma\in \mathcal{H}$   and  any   $v\notin  \sigma$,  we  have  that     $\sigma\sqcup \{v\}$  is    a  hyperedge  in  $\mathcal{H}$.   Since  
  \begin{eqnarray*}
 dv (\sigma)=
 \begin{cases}
 0,   & {\rm~if~} v\in \sigma,\\
  \pm  (\sigma\sqcup\{v\}),    &  {\rm~if~}  v\notin \sigma, 
 \end{cases}
 \end{eqnarray*}
  we  have    the  induced  homomorphism   $dv:   R_{n }(\mathcal{H})\longrightarrow   R_{n+1}(\mathcal{H})$  for  any  $v\in  V$  and  any  $n\in \mathbb{N}$.

  Conversely,  suppose   we  have    the  induced  homomorphism   $dv:   R_{n }(\mathcal{H})\longrightarrow   R_{n+1}(\mathcal{H})$  for  any  $v\in  V$  and  any  $n\in \mathbb{N}$.   Then  for  any  $\sigma\in \mathcal{H}$
   and  any  $v\in V\setminus \sigma$,   we   have  that  $\sigma\sqcup  \{v\} $  is  a  hyperedge  in  $ \mathcal{H}$.    By  induction,  for  any  finite   superset  $\tau\supseteq\sigma$  such  that   $\tau\subseteq   V$,  we have $\tau\in \mathcal{H}$.    Therefore,  $\mathcal{H}$  is  an   independence  hypergraph.  
 \end{proof}
 
 \begin{definition}\label{def-22222}
 Let  $\mathcal{H}$  be  a  hypergraph  with  vertices  from  $V$.   Let  $v\in  V$.  
 We  say that  $\frac{\partial}{\partial  v}$  is  {\it   $\mathcal{H}$-admissible}   if  (\ref{eq-0086a})    is  well-defined  for  any  $n\in \mathbb{N}$
 and  say  that  $dv$  is  {\it   $\mathcal{H}$-admissible}   if  (\ref{eq-0086b})   is  well-defined  for  any  $n\in \mathbb{N}$.   We  use  $V_\mathcal{H}$  to  denote  the  collection of  all  $v\in  V$  such  that  $\frac{\partial}{\partial  v}$  is   $\mathcal{H}$-admissible   and    use  $V^\mathcal{H}$  to  denote  the  collection of  all  $v\in  V$  such  that  $dv$  is   $\mathcal{H}$-admissible.   
 \end{definition}
 

 \begin{theorem}[Main  Result  I]
 \label{th-3.2ccv}
 Let  $\mathcal{H}$  be  a  hypergraph  with  its  vertices  from  $V$.  
 \begin{enumerate}[(1).]
 \item
 For    any    $w:  V_\mathcal{H}\longrightarrow  R$    we  have  a  chain  complex 
\begin{eqnarray*}
\xymatrix{
\cdots \ar[r]^-{\delta_{n+1}(w)} &{\rm  Ext}_{n }(V_\mathcal{H})\ar[r]^-{\delta_n(w)}  & {\rm  Ext}_{n-1 }(V_\mathcal{H})
\ar[r]^-{\delta_{n-1}(w)} & \cdots \ar[r]^-{\delta_{2}(w)} &{\rm  Ext}_{1 }(V_\mathcal{H}) \ar[r]^-{\delta_{1}(w)}  & R\ar[r]^-{\delta_{0}(w)}   &0.    
} 
\end{eqnarray*}
Moreover,  if   $V_\mathcal{H}$  is  a  finite  set  and   $w$  is  non-vanishing  on  $V_\mathcal{H}$,    then  the     chain complex   is  a    long  exact  sequence;   
\item
For    any    $w:  V^\mathcal{H}\longrightarrow  R$    we  have  a  chain  complex 
\begin{eqnarray*}
\xymatrix{
\cdots \ar[r]^-{\delta^{n+1}(w)} &{\rm  Ext}^{n }(V^\mathcal{H})\ar[r]^-{\delta^n(w)}  & {\rm  Ext}^{n-1 }(V^\mathcal{H})
\ar[r]^-{\delta^{n-1}(w)} & \cdots  \ar[r]^-{\delta^{2}(w)} &{\rm  Ext}^{1 }(V^\mathcal{H}) \ar[r]^-{\delta^{1}(w)}  & R\ar[r]^-{\delta^{0}(w)}  &0.
}
\end{eqnarray*}
Moreover,  if   $V^\mathcal{H}$  is  a  finite  set  and   $w$  is  non-vanishing  on  $V^\mathcal{H}$,    then  the     chain complex   is  a    long  exact  sequence.  
\end{enumerate}
 \end{theorem}
 
 \begin{proof}
 By  Definition~\ref{def-22222},   both  ${\rm  Ext}_*(V_\mathcal{H})$  and  $\delta_n(w):  {\rm  Ext}_n(V_\mathcal{H})\longrightarrow  {\rm  Ext}_{n-1}(V_\mathcal{H})$   are   well-defined.   
 Substituting  $V$  in   the  first  chain  complex  in  Proposition~\ref{pr-88.8}  with  $V_\mathcal{H}$,    
 we  obtain   (1).   Similarly,  both  ${\rm  Ext}^*(V^\mathcal{H})$  and  $\delta^n(w):  {\rm  Ext}^n(V^\mathcal{H})\longrightarrow  {\rm  Ext}^{n-1}(V^\mathcal{H})$   are   well-defined.   Substituting   $V$  in   the  second  chain  complex  in  Proposition~\ref{pr-88.8}  with  $V^\mathcal{H}$,     
 we  obtain   (2).   
 \end{proof}
 
 \begin{definition}
 We  call  the  chain  complexes  in  Theorem~\ref{th-3.2ccv}  the  {\it   Kouzul   complexes}  of    $\mathcal{H}$  with  respect to  $w$   and denote them  by  $K_*(\mathcal{H},w;R)$  and  $K^*(\mathcal{H},w;R)$  respectively.  
 \end{definition}
 
 The  next  two  corollaries  follow  from  Lemma~\ref{le-3.1ccv}  and  Theorem~\ref{th-3.2ccv}. 
 
 \begin{corollary}\label{co-3.3mkw}
  Let  $\mathcal{K}$  be  a  simplicial  complex  with  its  vertices  from  $V$.  For    any    $w:  V\longrightarrow  R$,   $K_*(\mathcal{K},w;R)$   is   the  chain  complex 
\begin{eqnarray*}
&\xymatrix{
\cdots \ar[r]^-{\delta_{n+1}(w)} &{\rm  Ext}_{n }(V)\ar[r]^-{\delta_n(w)}  & {\rm  Ext}_{n-1 }(V)
\ar[r]^-{\delta_{n-1}(w)} & \cdots \ar[r]^-{\delta_{2}(w)} &{\rm  Ext}_{1 }(V) \ar[r]^-{\delta_{1}(w)}  & R\ar[r]^-{\delta_{0}(w)}   &0. 
}
\end{eqnarray*}
Moreover,  if   $V$  is  a  finite  set  and   $w$  is  non-vanishing,    then    $K_*(\mathcal{K},w;R)$    is  a    long  exact  sequence. 
 \end{corollary}
 
 \begin{proof}
 By  Lemma~\ref{le-3.1ccv}~(1),  we  have  $V_\mathcal{K}=V$.  The  corollary follows  from  Theorem~\ref{th-3.2ccv}~(1).  
 \end{proof}
 
  \begin{corollary}\label{co-3.6mkw}
  Let  $\mathcal{L}$  be  an   independence   hypergraph  with  its  vertices  from  $V$.  
For    any    $w:  V\longrightarrow  R$,      $K^*(\mathcal{L},w;R)$    is   the   chain  complex 
\begin{eqnarray*}
&\xymatrix{
\cdots \ar[r]^-{\delta^{n+1}(w)} &{\rm  Ext}^{n }(V)\ar[r]^-{\delta^n(w)}  & {\rm  Ext}^{n-1 }(V)
\ar[r]^-{\delta^{n-1}(w)} & \cdots \ar[r]^-{\delta^{2}(w)} &{\rm  Ext}^{1 }(V) \ar[r]^-{\delta^{1}(w)}  & R\ar[r]^-{\delta^{0}(w)}   &0. 
}
\end{eqnarray*}
Moreover,  if   $V$  is  a  finite  set  and   $w$  is  non-vanishing,    then   $K^*(\mathcal{L},w;R)$   is  a    long  exact  sequence. 
 \end{corollary}
 
   \begin{proof}
  By  Lemma~\ref{le-3.1ccv}~(2),       we  have  $V^\mathcal{L}=V$.  The  corollary follows  from  Theorem~\ref{th-3.2ccv}~(2).  
 \end{proof}
 
 \begin{definition}
 Let  $\mathcal{H}$   be   a  hypergraph  with  vertices  from  $V$   and   let   $\mathcal{H}'$  be  a   hypergraph   with    vertices  from  $V'$.  A  {\it  morphism}  of  hypergraphs  $\varphi:  \mathcal{H}\longrightarrow  \mathcal{H}'$  is  a  map  $\varphi:  V\longrightarrow  V'$ 
 such  that  for  any  $\sigma\in \mathcal{H}$,  it  holds   $f(\sigma)\in  \mathcal{H}'$  where  $f(\sigma)$  is   given  by  $f(\sigma)=\{f(v)\mid  v\in  \sigma\}$.    In  particular,  
 \begin{enumerate}[(1).]
 \item
 if  both  $\mathcal{H}$  and  $\mathcal{H}'$  are  simplicial  complexes,  then  $\varphi$  is  called  a  {\it  simplicial  map};  
 
 \item
 if  both  $\mathcal{H}$  and  $\mathcal{H}'$  are  independence  hypergraphs,  then  $\varphi$  is  called  a  {\it  morphism  of  independence hypergraphs}.  
 \end{enumerate}
 \end{definition}
 
 \begin{remark}
 Let  ${\bf H}$  be the  category    whose  objects  are  hypergraphs and whose  morphisms are  morphisms of  hypergraphs.  Let  ${\bf  K}$  be  the  category  whose  objects  are  simplicial  complexes and whose  morphisms are  simplicial maps.  Let   ${\bf  L}$  be  the category  whose  objects  are  independence  hypergraphs
  and  whose  morphisms are  morphisms  of  independence  hypergraphs.  
  It  is  direct  that    both  ${\bf  K}$  and  ${\bf  L}$  are  full  sub-categories  of  ${\bf  H}$.  
  \end{remark}
 
 \begin{proposition}\label{pr-kkh881}
 Let  $\varphi: \mathcal{K}\longrightarrow \mathcal{K}'$  be  a  simplicial map      induced  by  a  bijection 
   $\varphi:  V\longrightarrow  V$   of  the  vertices.    Suppose we  have  a  map  $w:  V\longrightarrow  R$  such  that  there  exists  an  endomorphism  of  rings  $ f_{\varphi,w}:  R\longrightarrow  R$   sending   $w(v)$  to  
 $w(\varphi(v))$  for  any  $v\in  V$.     Then   we  have   an   induced    chain  map   
 \begin{eqnarray}
 {\rm  Ext}_*(\varphi): ~~~  K_*(\mathcal{K},w;R)\longrightarrow  K_*(\mathcal{K}',w;R).   
  \label{eq-abcdz1}
 \end{eqnarray}
 Moreover,   if   $V$  is  a  finite  set  and   $w$  is  non-vanishing,  then     (\ref{eq-abcdz1})
  is  a    morphism   of  long  exact  sequences.  
 \end{proposition}
 
 \begin{proof}
 By  Corollary~\ref{co-3.3mkw},     $V_\mathcal{K}= V_{\mathcal{K}'} =  V$.   
 Thus  by  Proposition~\ref{pr-068a},  we  have  a  chain  map  (\ref{eq-abcdz1}).    If   $V$  is  a  finite  set  and   $w$  is  non-vanishing,  then  (\ref{eq-abcdz1})  will  be  a  morphism   of  long  exact  sequences.  
 \end{proof}
 
  \begin{proposition}\label{pr-kkh882}
 Let  $\varphi: \mathcal{L}\longrightarrow \mathcal{L}'$  be  a  morphism  of  independence  hypergraphs       induced  by  a  bijection 
   $\varphi:  V\longrightarrow  V$   of  the  vertices.    Suppose we  have  a  map  $w:  V\longrightarrow  R$  such  that  there  exists  an  endomorphism  of  rings  $ f_{\varphi,w}:  R\longrightarrow  R$   sending   $w(v)$  to  
 $w(\varphi(v))$  for  any  $v\in  V$.     Then   we  have   an    induced    chain  map  
 \begin{eqnarray}
 {\rm  Ext}^*(\varphi): ~~~  K^*(\mathcal{L},w;R)\longrightarrow  K^*(\mathcal{L}',w;R).   
  \label{eq-abcdz2}
 \end{eqnarray}
 Moreover,   if   $V$  is  a  finite  set  and   $w$  is  non-vanishing,  then     (\ref{eq-abcdz2})  is  a  morphism   of  long  exact  sequences.  
 \end{proposition}
 
 \begin{proof}
  By  Corollary~\ref{co-3.3mkw},     $V^\mathcal{L}= V^{\mathcal{L}'} =  V$.   
 Thus  by  Proposition~\ref{pr-068a},  we  have  a  chain  map  (\ref{eq-abcdz2}).    If   $V$  is  a  finite  set  and   $w$  is  non-vanishing,  then  (\ref{eq-abcdz2})  will  be  a  morphism   of  long  exact  sequences.  
 \end{proof}

 For  any  set  $X$,  we  use  $R(X)$  to  denote  the  free  $R$-module  generated  by    the  elements  in   $X$.  
 Each  element   in  $R(X)$  is  a  finite  linear  combination  $\sum_{i=1}^n  r_i  x_i$,   where   $n$  is   a  positive  integer,      $r_i\in  R$  and  $x_i\in  X$.  
 For  any    subset  $A\subseteq  R$,  we  use  $\sum_{a\in  A} a  R$  to  denote  the  ideal generated  by  $A$   
 consisting  of  the  elements    of  the  form  
 $ a_1 r_1+ \cdots + a_n r_n$,  where    $ r_1,\ldots, r_n\in  R$,  $a_1,\ldots, a_n\in  A$   and    $n$  is  
 a  positive   integer.   
 
 \begin{example}\label{ex-vmqwe01}
 Let  $V=\{v_0,v_1\}$.   Let  $w:  V\longrightarrow  R$.   Then  $\Delta[V]=\{\{v_0\},  \{v_1\},  \{v_0,v_1\}\}$   is  a  simplicial  complex  and  is  also  
 an  independence  hypergraph.  
 
 \begin{enumerate}[(1).]
 \item
  By  Theorem~\ref{th-3.2ccv}~(1)  or Corollary~\ref{co-3.3mkw},   the  chain  complex  $K_*(\Delta[V],w;R)$  is 
 \begin{eqnarray*}
 \xymatrix{
 0\ar[r]  &  R\Big(\dfrac{\partial}{\partial  v_0}\wedge\dfrac{\partial}{\partial  v_1}\Big)\ar[r]^-{\delta_2(w)}  &  R\Big(\dfrac{\partial}{\partial  v_0}\Big)\oplus  R\Big(\dfrac{\partial}{\partial  v_1}\Big)
 \ar[r]^-{\delta_1(w)}  &   w(v_0)   R  +  w(v_1)  R \ar[r]^-{\delta_0(w)}&   0,  
 }
 \end{eqnarray*} 
 where  
\begin{eqnarray*}
&\delta_2(w) \Big(\dfrac{\partial}{\partial  v_0}\wedge\dfrac{\partial}{\partial  v_1}\Big) =  w(v_0) \dfrac{\partial}{\partial  v_1}
- w(v_1)\dfrac{\partial}{\partial  v_0},  \\
&\delta_1(w) \Big(\dfrac{\partial}{\partial  v_0}\Big)  =   w(v_0),\\     
&\delta_1(w) \Big(\dfrac{\partial}{\partial  v_1}\Big) =   w(v_1).  
\end{eqnarray*}
Suppose  in  addition that  $w$  is  non-vanishing.  
Then 
        $K_*(\Delta[V],w;R)$   is  an  exact  sequence   such  that   
\begin{eqnarray*}
{\rm  Ker} \delta_1(w)  = {\rm  Im}  \delta_2(w)  =  R\Big( w(v_0) \frac{\partial}{\partial  v_1}
- w(v_1)\frac{\partial}{\partial  v_0}\Big).  
\end{eqnarray*}
\item
By  Theorem~\ref{th-3.2ccv}~(2)  or  Corollary~\ref{co-3.6mkw},    the  chain  complex  $K^*(\Delta[V],w;R)$  is
 \begin{eqnarray*}
 \xymatrix{
 0\ar[r]  &   R(dv_0\wedge  dv_1)\ar[r]^-{\delta^2(w)}&   R(dv_0)\oplus  R(dv_1)
 \ar[r]^-{\delta^1(w)} &  w(v_0)   R  +  w(v_1)  R  \ar[r]^-{\delta^0(w)}&   0,  
 }
 \end{eqnarray*} 
 where  
\begin{eqnarray*}
&\delta^2(w) (d v_0 \wedge  d  v_1 ) =  w(v_0) d v_1 
- w(v_1)d v_0,  \\
&\delta^1(w) (d v_0 )  =   w(v_0),\\  
&\delta^1(w) ( d v_1 ) =   w(v_1).  
\end{eqnarray*}
Suppose  in  addition that  $w$  is  non-vanishing.   Then     $K^*(\Delta[V],w;R)$    is  an  exact  sequence   such  that   
\begin{eqnarray*}
{\rm  Ker} \delta^1(w)  = {\rm  Im}  \delta^2(w)  =  R ( w(v_0)   d v_1 
- w(v_1)  dv_0  ).  
\end{eqnarray*}
 \end{enumerate}
 \end{example}
 
 \begin{remark}
 Example~\ref{ex-vmqwe01}  is  a  discrete  analog of  \cite[Eq.  (3.56)]{siu}.  
 \end{remark}

  \begin{example}
  Let  $V=\{v_0,v_1,\ldots, v_n\}$,  $n\geq  2$.   Let  $w:  V\longrightarrow  R$.     Let  $1\leq   m\leq  n-1$.  Let  $U=\{v_{i_0}, v_{i_1},\ldots,v_{i_m}\}$,  where  $0\leq  i_0<i_1<\cdots<i_m\leq  n$,   be  a  fixed  subset  of   $V$.     
    Write  $V\setminus   U=\{v_{j_1}, \ldots,  v_{j_{n-m}}\}$.  
  \begin{enumerate}[(1).]
  \item
  Take the  simplicial  complex  
  \begin{eqnarray*}
  \mathcal{K}=\{\{v_i\}\mid  0\leq  i\leq  n\}\cup\{  \{v_i,v_j\}\mid    0\leq  i<j\leq  n\}\cup \{ \{v_i,v_j,v_k\}\mid   0\leq  i<j<k\leq  n \}.
  \end{eqnarray*} 
  By  Corollary~\ref{co-3.3mkw},    the  chain  complex  $K_*(\mathcal{K},w;R)$  is
 \begin{eqnarray*}
 &0\longrightarrow  R\Big(\dfrac{\partial}{\partial  v_i}\wedge\dfrac{\partial}{\partial  v_j}\wedge\dfrac{\partial}{\partial  v_k}\mid  0\leq  i<j<k\leq  n \Big)\overset{\delta_3(w)}{\longrightarrow}  R\Big(\dfrac{\partial}{\partial  v_i}\wedge\dfrac{\partial}{\partial  v_j} \mid  0\leq  i<j \leq  n \Big) \\
& \overset{\delta_2(w)}{\longrightarrow} R\Big(\dfrac{\partial}{\partial  v_i}\mid  0\leq  i\leq  n\Big) 
 \overset{\delta_1(w)}{\longrightarrow}   \sum_{i=0}^n  w(v_i)   R  \overset{\delta^0(w)}{\longrightarrow}   0.   
 \end{eqnarray*} 
 In  addition,  if  $w$   is   non-vanishing,  then    $K_*(\mathcal{K},w;R)$  is  an  exact  sequence such  that 
\begin{eqnarray*}
 {\rm  Ker} \delta_2(w)= {\rm    Im } \delta_3(w) &=& R\Big( w(v_i) \dfrac{\partial}{\partial  v_j}\wedge\dfrac{\partial}{\partial  v_k}- w(v_j) \dfrac{\partial}{\partial  v_i}\wedge\dfrac{\partial}{\partial  v_k} + w(v_k) \dfrac{\partial}{\partial  v_i}\wedge\dfrac{\partial}{\partial  v_j} \\
&&\mid   0\leq  i<j<k\leq  n \Big), \\
 {\rm  Ker} \delta_1(w)= {\rm    Im } \delta_2(w) &=& R\Big( w(v_i) \dfrac{\partial}{\partial  v_j}-w(v_j) \dfrac{\partial}{\partial  v_i}\mid  0\leq  i<j\leq  n\Big).  
\end{eqnarray*}
  \item
  Take  the  independence  hypergraph
   \begin{eqnarray*}
  \mathcal{L}&=&\{\{v_0,v_1,\ldots,v_n\} \}\cup\{  \{v_0,\ldots,\widehat{v_i},\ldots,v_n\}\mid    0\leq  i\leq  n\}\\
  &&\cup \{  \{v_0,\ldots,\widehat{v_i},\ldots,\widehat{v_j},\ldots, v_n\}\mid   0\leq  i<j \leq  n \}\\
  &&\cup \{  \{v_0,\ldots,\widehat{v_i},\ldots,\widehat{v_j},\ldots,\widehat{v_k}, \ldots, v_n\}\mid   0\leq  i<j<k \leq  n \}.
  \end{eqnarray*} 
    By   Corollary~\ref{co-3.6mkw},       the  chain  complex  $K^*(\mathcal{L},w;R)$  is  
  \begin{eqnarray*}
 &0\longrightarrow  R (d  v_i \wedge   d v_j \wedge   d  v_k  \mid  0\leq  i<j<k\leq  n  )\overset{\delta^3(w)}{\longrightarrow}  R ( d  v_i \wedge    d  v_j  \mid  0\leq  i<j \leq  n   ) \\
& \overset{\delta^2(w)}{\longrightarrow} R (  d v_i \mid  0\leq  i\leq  n  ) 
 \overset{\delta^1(w)}{\longrightarrow}   \sum_{i=0}^n  w(v_i)   R  \overset{\delta^0(w)}{\longrightarrow}   0.   
 \end{eqnarray*} 
  In  addition,  if  $w$   is   non-vanishing,  then    $K^*(\mathcal{L},w;R)$  is  an  exact  sequence such  that 
\begin{eqnarray*}
 {\rm  Ker} \delta^2(w)= {\rm    Im } \delta^3(w) &=& R ( w(v_i)   d v_j \wedge  d v_k - w(v_j) d  v_i \wedge
    d  v_k  + w(v_k) d v_i \wedge  d  v_j  \\
&&\mid   0\leq  i<j<k\leq  n  ), \\
 {\rm  Ker} \delta^1(w)= {\rm    Im } \delta^2(w) &=& R ( w(v_i)   d v_j -w(v_j) d  v_i \mid  0\leq  i<j\leq  n ).  
\end{eqnarray*}

  \item
   Take  the  $(m+1)$-uniform  hypergraph  
  \begin{eqnarray*}
  \mathcal{H}  = \{\{v_{i_0}, v_{i_1},\ldots,v_{i_m}\}\mid  0\leq  i_0<i_1<\cdots<i_m\leq  n \}.  
  \end{eqnarray*}
  Then  $V_{\mathcal{H} }=V^{\mathcal{H} }=\emptyset$.   The  chain  complexes      $K_*(\mathcal{H},w;R)$  and  
  $K^*(\mathcal{H},w;R)$    are  trivial.  
  \item
  Take  the  hypergraph 
  \begin{eqnarray*}
  \mathcal{H} = \{\sigma\in \Delta[V]\mid  \sigma\cap  U\neq  \emptyset\}.  
    \end{eqnarray*}
    We  have  $V_\mathcal{H}=V\setminus  U$  and     $V^\mathcal{H}=V$.      
     By  Theorem~\ref{th-3.2ccv}~(1),    the  chain  complex  $K_*(\mathcal{H},w;R)$  is 
  \begin{eqnarray*}
   &0\longrightarrow  R\Big(\dfrac{\partial}{\partial  v_{j_1}}\wedge\cdots\wedge\dfrac{\partial}{\partial  v_{j_{n-m}}}\Big)
   \\
   &\overset{\delta_{n-m}(w)}{\longrightarrow} R\Big(\dfrac{\partial}{\partial  v_{j_1}}\wedge\cdots\wedge\widehat{\dfrac{\partial}{\partial  v_{j_l}}}\wedge\cdots\wedge \dfrac{\partial}{\partial  v_{j_{n-m}}}\mid  1\leq  l\leq  n-m\Big)\overset{\delta_{n-m-1 }(w)}{\longrightarrow} \cdots  \\
    &      \overset{\delta_{2}(w)}{\longrightarrow} 
   R\Big(\dfrac{\partial}{\partial  v_{j_l}}\mid  1\leq  l\leq  n-m\Big) \overset{\delta_{1}(w)}{\longrightarrow} \sum_{l=1}^{n-m}  w(v_{j_l})  R  \overset{\delta_{0}(w)}{\longrightarrow}  0.  
  \end{eqnarray*} 
  If  in  addition  that $w$  is  non-vanishing  on  $V\setminus   U$,  then   $K_*(\mathcal{H},w;R)$  is  a  long  exact  sequence.   
On the other  hand,  by  Theorem~\ref{th-3.2ccv}~(2),    the  chain  complex  $K^*(\mathcal{H},w;R)$  is
  \begin{eqnarray*}
   &0\longrightarrow  R (d  v_{ 0} \wedge\cdots\wedge   d v_{n}  )
   \overset{\delta^{n+1}(w)}{\longrightarrow} R (  d v_{ 0} \wedge\cdots\wedge\widehat{   dv_{i }} \wedge\cdots\wedge  d v_{n} \mid  0\leq  i\leq  n )\\
    &   \overset{\delta^{n}(w)}{\longrightarrow}  \cdots    \overset{\delta^{2}(w)}{\longrightarrow} 
   R (d  v_{i } \mid  0\leq   i\leq   n ) \overset{\delta^{1}(w)}{\longrightarrow} \sum_{i=0}^n   w(v_{i })  R  \overset{\delta^{0}(w)}{\longrightarrow}  0.  
  \end{eqnarray*} 
 If in  addition that  $w$  is  non-vanishing  on  $V$,  then    $K^*(\mathcal{H},w;R)$  is  a  long  exact  sequence.

 \item 
 Take  the  hypergraph 
  \begin{eqnarray*}
  \mathcal{H}  = \{\sigma\in \Delta[V]\mid  \sigma\cap U=\emptyset\}.  
    \end{eqnarray*}
    We  have     $V_\mathcal{H}=V$.  
    By  Theorem~\ref{th-3.2ccv}~(1),    the  chain  complex  $K_*(\mathcal{H},w;R)$  is 
  \begin{eqnarray*}
   &0\longrightarrow  R\Big(\dfrac{\partial}{ \partial   v_{ 0}} \wedge\cdots\wedge  \dfrac{\partial}{  \partial   v_{n}} \Big)
   \overset{\delta_{n+1}(w)}{\longrightarrow} R\Big(  \dfrac{\partial}{ \partial  v_{ 0}} \wedge\cdots\wedge\widehat{   \dfrac{\partial}{\partial  v_{i }}} \wedge\cdots\wedge  \dfrac{\partial}{\partial  v_{n}} \mid  0\leq  i\leq  n\Big)\\
    &   \overset{\delta_{n}(w)}{\longrightarrow}  \cdots    \overset{\delta_{2}(w)}{\longrightarrow} 
   R\Big(\dfrac{\partial}{  \partial   v_{i } }\mid  0\leq   i\leq   n\Big) \overset{\delta_{1}(w)}{\longrightarrow} \sum_{i=0}^n   w(v_{i })  R  \overset{\delta_{0}(w)}{\longrightarrow}  0.  
  \end{eqnarray*} 
 If  in addition  that  $w$  is  non-vanishing  on  $V$,  then       $K_*(\mathcal{H},w;R)$  is  a  long  exact  sequence. 
 On the other  hand,  by  Theorem~\ref{th-3.2ccv}~(2),    the  chain  complex  $K^*(\mathcal{H},w;R)$  is 
  \begin{eqnarray*}
   &0\longrightarrow  R (d  v_{j_1} \wedge\cdots\wedge   d  v_{j_{n-m}}  )
   \overset{\delta^{n-m}(w)}{\longrightarrow} R ( d  v_{j_1} \wedge\cdots\wedge\widehat{  d  v_{j_l}} \wedge\cdots\wedge    d  v_{j_{n-m}} \mid  1\leq  l\leq  n-m )\\
    &   \overset{\delta^{n-m-1 }(w)}{\longrightarrow}  \cdots    \overset{\delta^{2}(w)}{\longrightarrow} 
   R (  d  v_{j_l} \mid  1\leq  l\leq  n-m ) \overset{\delta^{1}(w)}{\longrightarrow} \sum_{l=1}^{n-m}  w(v_{j_l})  R  \overset{\delta_{0}(w)}{\longrightarrow}  0.  
  \end{eqnarray*} 
   If  in  addition  that $w$  is  non-vanishing  on  $V\setminus   U$,  then    $K^*(\mathcal{H},w;R)$  is  a  long  exact  sequence.   
  \end{enumerate} 
 \end{example}
 
 \begin{example}
 Let  $V=\mathbb{Z}$.   Let  $w:  \mathbb{Z}\longrightarrow  R$.  
 \begin{enumerate}[(1).]
 \item
 Take  the  simplicial  complex
 \begin{eqnarray*}
 \mathcal{K}= \{\{v_n\}\mid  n\in \mathbb{Z}\}\cup    \{\{v_n,v_{n+1}\}\mid  n\in \mathbb{Z}\}\cup  \{\{v_n,v_{n+1},v_{n+2}\}\mid  n\in \mathbb{Z}\}.  
  \end{eqnarray*}
 By  Corollary~\ref{co-3.3mkw},    the  chain  complex  $K_*(\mathcal{K},w;R)$  is 
 \begin{eqnarray*}
 &0\longrightarrow  R\Big(\dfrac{\partial}{\partial  v_i}\wedge\dfrac{\partial}{\partial  v_j}\wedge\dfrac{\partial}{\partial  v_k}\mid    i<j<k,  i,j,k\in \mathbb{Z}  \Big)\overset{\delta_3(w)}{\longrightarrow}  R\Big(\dfrac{\partial}{\partial  v_i}\wedge\dfrac{\partial}{\partial  v_j} \mid     i<j,   i,j\in\mathbb{Z} \Big) \\
& \overset{\delta_2(w)}{\longrightarrow} R\Big(\dfrac{\partial}{\partial  v_i}\mid    i\in \mathbb{Z} \Big) 
 \overset{\delta_1(w)}{\longrightarrow}   \Big\{\sum_{i\in  \sigma}   w(v_i)   R \mid   \sigma\subset  \mathbb{Z}{\rm~are~finite~sets} \Big\}\overset{\delta_0(w)}{\longrightarrow}   0.   
 \end{eqnarray*}

 \item
 Recall that  a  hyperedge  with  vertices from $\mathbb{Z}$  is  just a non-empty  finite  subset of  $\mathbb{Z}$. 
Let  $\sigma$  be  a  hyperedge  with  vertices from  $\mathbb{Z}$.  Take  the  independence  hypergraph  
  \begin{eqnarray*}
 \mathcal{L}= \{ \tau\supseteq \sigma \mid   \tau{\rm~is~a~hyperedge~with~vertices~from~}\mathbb{Z}\}.  
  \end{eqnarray*}
 By  Corollary~\ref{co-3.6mkw},   the  chain  complex  $K^*(\mathcal{L},w;R)$  is
 \begin{eqnarray*}
&\cdots \overset{\delta^{n+1}(w)}{\longrightarrow} R (d v_{i_1} \wedge   \cdots \wedge   d  v_{i_n} \mid    i_1<\cdots <i_n,  i_1,\cdots,i_n\in \mathbb{Z},  v_{i_1},\cdots,v_{i_n}\notin\sigma  )\overset{\delta^n(w)}{\longrightarrow}\\
&R (d v_{i_1} \wedge   \cdots \wedge   d  v_{i_{n-1}} \mid    i_1<\cdots <i_{n-1},  i_1,\cdots,i_{n-1}\in \mathbb{Z},  v_{i_1},\cdots,v_{i_{n-1}}\notin\sigma  )\overset{\delta^{n-1}(w)}{\longrightarrow}\\
&\cdots 
 \overset{\delta^2(w)}{\longrightarrow} R (d v_i \mid    i\in \mathbb{Z},  v_i\notin  \sigma  ) 
 \overset{\delta^1(w)}{\longrightarrow}   \Big\{\sum_{i\in  \eta}   w(v_i)   R \mid   \eta\subset  \mathbb{Z}\setminus \sigma {\rm~are~finite~sets} \Big\}\overset{\delta^0(w)}{\longrightarrow}   0.
 \end{eqnarray*}
 \end{enumerate}
 \end{example}
 
 \begin{example}
 Let  $V=\{v_0,v_1,v_2\}$.  Let   $w:  V\longrightarrow   R$.    Let  $S_3$  be  the  permutation  group  on  $\{0,1,2\}$.  
 Let  $   s \in   S_3$. 
 \begin{enumerate}[(1).]
 \item
 Take the  simplicial complexes  
 \begin{eqnarray*}
 \mathcal{K}&=&\{\{v_0\},  \{v_1\},    \{v_0,v_1\}\},\\
 \mathcal{K}'&=&\{\{v_0\},  \{v_1\},  \{v_2\},  \{v_0,v_1\},  \{v_1,v_2\}, \{v_0,v_2\}\}
 \end{eqnarray*}
 and  the  simplicial  map  $s:  \mathcal{K}\longrightarrow  \mathcal{K}'$  induced  by  the  permutation  $s$ 
  on  $V$,  where $s\in  S_3$  is  arbitrarily  chosen.   The   chain  complex    $K_*(\mathcal{K},w;R)$  
  is 
  \begin{eqnarray*}
  \xymatrix{
 R\Big(\dfrac{\partial}{\partial  v_0}\wedge \dfrac{\partial}{\partial  v_1} ) 
  \ar[r]^-{\delta_2(w)}  &  R\Big(\dfrac{\partial}{\partial  v_i}\mid  i=0,1 \Big) 
   \ar[r]^-{\delta_1(w)}   &     w(v_0)   R + w(v_1)  R \ar[r]^-{\delta_0(w)}    &0, 
  }
  \end{eqnarray*}
  the  chain  complex   $K_*(\mathcal{K}',w;R)$  is  
  \begin{eqnarray*}
  \xymatrix{
   R\Big(\dfrac{\partial}{\partial  v_0}\wedge \dfrac{\partial}{\partial  v_1}\wedge  \dfrac{\partial}{\partial  v_2} \Big)
  \ar[r]^-{\delta_3(w)}  & R\Big(\dfrac{\partial}{\partial  v_i}\wedge \dfrac{\partial}{\partial  v_j} \mid  0\leq  i<j\leq  2\Big) 
  \ar[r]^-{\delta_2(w)}  &  R\Big(\dfrac{\partial}{\partial  v_i}\mid  i=0,1,2  \Big)\\
   \ar[r]^-{\delta_1(w)}   &  \sum_{i=0}^2  w(v_i)   R \ar[r]^-{\delta_0(w)}    &0  
  }
  \end{eqnarray*}
  and  the    induced    chain  map     in  (\ref{eq-abcdz1})   is     given  by
  \begin{eqnarray*}
  {\rm  Ext}_*(s)(\frac{\partial}{\partial  v_i}) =\frac{\partial}{\partial   v_{s(i)}},~~~  i=0,1.   
  \end{eqnarray*}
  
 \item
 Take  the  independence  hypergraphs
 \begin{eqnarray*}
 \mathcal{L} &=& \{\{v_0,v_1,v_2\},  \{v_0,v_1\}\}, \\
 \mathcal{L}'&=& \{\{v_0,v_1,v_2\},  \{v_0,v_1\},  \{v_0\}, \{v_1\},  \{v_1,v_2\}\}   
 \end{eqnarray*}
 and  the  morphism  $s:  \mathcal{L}\longrightarrow \mathcal{L}'$  where  
 \begin{eqnarray*}
 s= \begin{pmatrix}
    0 &1 & 2 \\
    1 &2 & 0
  \end{pmatrix},  
  \begin{pmatrix}
    0 &1 & 2 \\
    2 &1 & 0
  \end{pmatrix}, 
    \begin{pmatrix}
    0 &1 & 2 \\
    1 &0 & 2
  \end{pmatrix}  
  {\rm ~or~} 
\begin{pmatrix}
    0 &1 & 2 \\
    0 &1 & 2
  \end{pmatrix}.   
 \end{eqnarray*}
 The   chain  complex    $K^*(\mathcal{L},w;R)$  as  well  as   $K^*(\mathcal{L}',w;R)$  is  
  \begin{eqnarray*}
  \xymatrix{
  &R (d  v_0 \wedge  d v_1 \wedge  d  v_2   )
  \ar[r]^-{\delta^3(w)}  & R (d  v_i \wedge   d  v_j  \mid  0\leq  i<j\leq  2  ) 
  \ar[r]^-{\delta^2(w)}  &  R  (d v_i \mid  i=0,1,2   )\\
  &\ar[r]^-{\delta^1(w)}   &  \sum_{i=0}^2  w(v_i)   R \ar[r]^-{\delta^0(w)}    &0  
  }
  \end{eqnarray*}
  and  the    induced    chain  map     in  (\ref{eq-abcdz2})   is     given  by
  \begin{eqnarray*}
  {\rm  Ext}^*(s)(d v_i ) =d v_{s(i)},~~~  i=0,1,2.  
  \end{eqnarray*}
 \end{enumerate}
 \end{example}
 
 \section{Constrained   (co)Homology  for  Hypergraphs}\label{s4}
 
  Let  $t\in  \mathbb{N}$.  
  Let  $0\leq q\leq 2t$  be  an  integer.    Let  $m\in \mathbb{Z}$.  There exist   a unique  $\lambda\in \mathbb{Z}$   and a unique  integer  $0\leq q\leq 2t$  such that $m=\lambda(2t+1)+q$.    Let   $\alpha\in {\rm Ext}_{2t+1}(V)$.  
  Let  $\omega\in {\rm Ext}^{2t+1}(V)$.    Let  $s\in  \mathbb{N}$.   Let $\beta\in {\rm Ext}_{2s}(V)$. 
   Let $\mu\in {\rm Ext}^{2s}(V)$.  
    Let  $s_1,s_2\in \mathbb{N}$.  Let $\beta_1\in    {\rm Ext}_{2s_1}(V) $   and   $\beta_2\in    {\rm Ext}_{2s_2}(V) $.  
     Let $\mu_1\in    {\rm Ext}^{2s_1}(V) $   and   $\mu_2\in    {\rm Ext}^{2s_2}(V) $. 
   
  \subsection{Constrained  homology  of  simplicial  complexes  and costrained  cohomology  of  independence  hypergraphs}

We  review the  constrained  homology  of  simplicial  complexes,  the  costrained  cohomology  of  independence  hypergraphs,  their  functorialities    and  the    Mayer-Vietoris  sequences  (cf.  \cite{camb2023,mv}).   
   Let   $\mathcal{K}$  be   a   simplicial  complex   with  its  vertices  from   $V$.       We  have a chain complex
\begin{eqnarray}\label{eq-vb345}
&\cdots\overset{\alpha}{\longrightarrow}
 R_{(n+\lambda)(2t+1)+q}(\mathcal{K})\overset{\alpha}{\longrightarrow}
 R_{(n-1+\lambda)(2t+1)+q}(\mathcal{K})\overset{\alpha}{\longrightarrow}
  \nonumber\\
&\cdots\overset{\alpha}{\longrightarrow}
 R_{(1+\lambda)(2t+1)+q}(\mathcal{K})\overset{\alpha}{\longrightarrow}
 R_{\lambda(2t+1)+q}(\mathcal{K})\overset{\alpha}{\longrightarrow}
 0,
\end{eqnarray}
denoted  by  $R_*(\mathcal{K},\alpha,m)$.   Let   $\mathcal{L}$  be  an   independence  hypergraph   with  its  vertices  from  $V$.   We  have  a  co-chain  complex   
\begin{eqnarray}\label{eq-vb368}
&\cdots \overset{\omega}{\longleftarrow}
 R_{(n+\lambda)(2t+1)+q}(\mathcal{L})\overset{\omega}{\longleftarrow}
 R_{(n-1+\lambda)(2t+1)+q}(\mathcal{L})\overset{\omega}{\longleftarrow}
 \nonumber  \\
&\cdots  \overset{\omega}{\longleftarrow} R_{(1+\lambda)(2t+1)+q}(\mathcal{L})\overset{\omega}{\longleftarrow}
 R_{\lambda(2t+1)q}(\mathcal{L})\overset{\omega}{\longleftarrow}
  0, 
\end{eqnarray}
denoted  by $R^*(\mathcal{L},\omega,m)$.    
The  $n$-th  {\it  constrained homology group} $H_n(\mathcal{K},\alpha,m)$ of $\mathcal{K}$  with respect to $\alpha$ and $m$   is  defined  to be the  $n$-th   homology group  of the chain complex $R_*(\mathcal{K},\alpha, m)$  and   the  $n$-th  {\it  constrained  cohomology  group} $H^n(\mathcal{L},\omega,m)$ of $\mathcal{L}$  with respect to $\omega$ and $m$   is  defined  to be the  $n$-th   cohomology group  of the co-chain complex $R^*(\mathcal{L},\omega, m)$    (cf.  \cite[Definition~4.3  and  Definition~4.4]{camb2023}).    
   We  have   an  induced   chain map
\begin{eqnarray*}
\beta:~~~  R_*(\mathcal{K},\alpha,m)\longrightarrow  R_*(\mathcal{K},\alpha,m-2s) 
\end{eqnarray*}
and  consequently  an  induced     homomorphism  of  the  constrained  homology groups
\begin{eqnarray*} 
\beta_*: ~~~H_n(\mathcal{K},\alpha,m)\longrightarrow  H_n(\mathcal{K},\alpha,m-2s), ~~~~~~  n\in  \mathbb{N}.  
\end{eqnarray*}
 We  have   an  induced   co-chain  map
\begin{eqnarray*}
\mu:~~~  R^*(\mathcal{L},\omega,m)\longrightarrow  R^*(\mathcal{L},\omega,m+2s) 
\end{eqnarray*}
 and  consequently  an  induced     homomorphism  of  the  constrained  cohomology groups
\begin{eqnarray*} 
\mu_*: ~~~H^n(\mathcal{L},\omega,m)\longrightarrow  H^n(\mathcal{L},\omega,m+2s), ~~~~~~  n\in  \mathbb{N}.  
\end{eqnarray*}
 The  commutative  diagram    of  chain  complexes   
 \begin{eqnarray*}
 \xymatrix{
 R_*(\mathcal{K},\alpha,m)\ar[r]^-{\beta_1} \ar[dd]_-{\beta_2} \ar[rdd] ^-{\beta_1\wedge \beta_2}&     R_*(\mathcal{K},\alpha,m-2s_1)\ar[dd]^-{\beta_2}\\
 \\
R_*(\mathcal{K},\alpha,m-2s_2)  \ar[r]^-{\beta_1} &     R_*(\mathcal{K},\alpha,m-2(s_1+s_2)) 
 }
 \end{eqnarray*}
induces  a  commutative diagram  of  constrained homology  groups
 \begin{eqnarray}\label{eq-diag-k}
 \xymatrix{
H_*(\mathcal{K},\alpha,m)\ar[r]^-{(\beta_1)_*} \ar[dd]_-{(\beta_2)_*}\ar[rdd] ^-{(\beta_1\wedge \beta_2)_*}&     H_*(\mathcal{K},\alpha,m-2s_1)\ar[dd]^-{(\beta_2)_*}\\
\\
H_*(\mathcal{K},\alpha,m-2s_2) \ar[r]^-{(\beta_1)_*}    &     H_*(\mathcal{K},\alpha,m-2(s_1+s_2)).  
 }
 \end{eqnarray}
 The  commutative  diagram    of  co-chain  complexes   
 \begin{eqnarray*}
 \xymatrix{
 R^*(\mathcal{L},\omega,m)\ar[r]^-{\mu_1}\ar[dd]_-{\mu_2} \ar[rdd] ^-{\mu_1\wedge \mu_2}&     R^*(\mathcal{L},\omega,m+2s_1)\ar[dd]^-{\mu_2}\\
 \\
 R^*(\mathcal{L},\omega,m+2s_2) \ar[r]^-{\mu_1}&     R^*(\mathcal{L},\omega,m+2(s_1+s_2)) 
 }
 \end{eqnarray*}
induces  a  commutative diagram  of  constrained  cohomology  groups
 \begin{eqnarray}\label{eq-diag-l}
 \xymatrix{
 H^*(\mathcal{L},\omega,m)\ar[r]^-{(\mu_1)_*}\ar[dd]_-{(\mu_2)_*} \ar[rdd] ^-{(\mu_1\wedge \mu_2)_*}&     H^*(\mathcal{L},\omega,m+2s_1)\ar[dd]^-{(\mu_2)_*}\\
 \\
 H^*(\mathcal{L},\omega,m+2s_2) \ar[r]^-{(\mu_1)_*}&    H^*(\mathcal{L},\omega,m+2(s_1+s_2)).  
 }
 \end{eqnarray}
 Let  $\mathcal{K}_1$  and  $\mathcal{K}_2$   be  
simplicial  complexes     and  $\mathcal{L}_1$  and  $\mathcal{L}_2$   be  
 independence  hypergraphs   with  their  vertices from  $V$.
 We  have  
 a  long  exact  sequence  of  the  constrained  homology  groups
 \begin{eqnarray}
\cdots \longrightarrow      H_n(\mathcal{K}_1 \cap  \mathcal{K}_2,\alpha, m)\longrightarrow 
  H_n (\mathcal{K}_1,\alpha, m) \oplus   H_n( \mathcal{K}_2,\alpha, m)\longrightarrow  \nonumber\\
  \longrightarrow  H_n(\mathcal{K}_1 \cup  \mathcal{K}_2,\alpha, m)\longrightarrow H_{n-1}(\mathcal{K}_1 \cap  \mathcal{K}_2,\alpha, m)\longrightarrow \cdots 
   \label{eq-lex1a}
 \end{eqnarray} 
 and  
 a  long  exact  sequence  of  the  constrained  cohomology  groups
 \begin{eqnarray}
\cdots \longrightarrow      H^n(\mathcal{L}_1 \cap  \mathcal{L}_2,\omega, m)\longrightarrow 
  H^n (\mathcal{L}_1,\omega, m) \oplus   H^n( \mathcal{L}_2,\omega, m)\longrightarrow  \nonumber \\
 \longrightarrow   H^n(\mathcal{L}_1 \cup  \mathcal{L}_2,\omega, m)\longrightarrow H^{n+1}(\mathcal{L}_1 \cap  \mathcal{L}_2,\omega, m)\longrightarrow \cdots.  
   \label{eq-lex2b}
 \end{eqnarray} 
 We  call  (\ref{eq-lex1a})  the  {\it  Mayer-Vietoris  sequence}  of  constrained  homology  of  simplicial complexes with respect to  $\alpha$  and  $m$  and  denote  it  as  ${\bf  MV}_*(\mathcal{K}_1,\mathcal{K}_2, \alpha, m)$.   We  call  (\ref{eq-lex2b})  the  {\it  Mayer-Vietoris  sequence}  of  constrained  cohomology  of   independence  hypergraphs   with respect to  $\omega$  and  $m$  and  denote  it  as  ${\bf  MV}^*(\mathcal{L}_1,\mathcal{L}_2, \omega, m)$.    
 The  next   lemma   is  a  slight  generalization   of  \cite[Proposition~3.4]{mv}  with   coefficients in  $R$.   
 
 \begin{proposition}\label{pr-5.a1}
 For  any     
simplicial  complexes  $\mathcal{K}_1$  and  $\mathcal{K}_2$   and  any    
 independence  hypergraphs   $\mathcal{L}_1$  and  $\mathcal{L}_2$      with  their  vertices from  $V$,
  \begin{enumerate}[(1).]
 \item
 the   commutative   diagram        (\ref{eq-diag-k})   and   the   long  exact   sequence  (\ref{eq-lex1a})       are  functorial   with respect  
to  simplicial  maps  induced  by  bijective  maps  between  the  vertices;   
 \item
     the   commutative  diagram        (\ref{eq-diag-l})  and   the   long  exact    sequence  (\ref{eq-lex2b})        are  functorial   with respect  
to  morphisms  of  independence  hypergraphs   induced  by  bijective  maps  between  the  vertices.  
\end{enumerate}
 \end{proposition}
 
 \begin{proof}
 (1).  
 Let  $\mathcal{K}_1$  and  $\mathcal{K}_2$  be  simplicial  complexes  with  vertices  from  $V$  
 and  let  $\mathcal{K}'_1$  and  $\mathcal{K}'_2$  be  simplicial  complexes  with  vertices  from  $V'$.   
 Suppose   $\varphi_1:  \mathcal{K}_1\longrightarrow  \mathcal{K}'_1$  and  $\varphi_2:  \mathcal{K}_2\longrightarrow  \mathcal{K}'_2$   are    simplicial  maps  induced  by  the   same  map     $\varphi:  V\longrightarrow  V'$   between  the  vertices.   
 Then  we  have  induced  simplicial maps 
 \begin{eqnarray*}
 \varphi_1\cap\varphi_2: &&  \mathcal{K}_1\cap  \mathcal{K}_2 \longrightarrow  \mathcal{K}'_1\cap  \mathcal{K}'_2,\\
  \varphi_1\cup\varphi_2: && \mathcal{K}_1\cup  \mathcal{K}_2 \longrightarrow  \mathcal{K}'_1\cup  \mathcal{K}'_2
 \end{eqnarray*}
 which  are  also  induced  by  $\varphi$.   Suppose  in addition  that  $\varphi$  is  bijective.  
 Then  
 \begin{eqnarray*}
 {\rm  Ext}_*(\varphi)(\alpha)\in  {\rm  Ext}_{2t+1}(V').
 \end{eqnarray*}      
  Consequently,    we  have      chain  maps    
 \begin{eqnarray}\label{eq-mblqo1}
 \varphi_{\#}: ~~~ R_*(\mathcal{K}_i,\alpha,m)\longrightarrow   R_*(\mathcal{K}'_i,{\rm  Ext}_*(\varphi)(\alpha),m)
 \end{eqnarray}
 for  $i=1,2$  as  well  as   
   chain  maps  
  \begin{eqnarray}\label{eq-mblqo2}
 \varphi_{\#}: && R_*(\mathcal{K}_1\cap\mathcal{K}_2,\alpha,m)\longrightarrow   R_*(\mathcal{K}'_1\cap \mathcal{K}'_2,{\rm  Ext}_*(\varphi)(\alpha),m),\\
 \label{eq-mblqo3}
  \varphi_{\#}: && R_*(\mathcal{K}_1\cup\mathcal{K}_2,\alpha,m)\longrightarrow   R_*(\mathcal{K}'_1\cup \mathcal{K}'_2,{\rm  Ext}_*(\varphi)(\alpha),m).  
 \end{eqnarray}
Applying   the  homology  functor  to the  chain  complexes  and  the  chain  maps  in 
(\ref{eq-mblqo1}),  (\ref{eq-mblqo2})  and  (\ref{eq-mblqo3}),   we  have  that   the   long  exact   sequence  (\ref{eq-lex1a})      is   functorial   with respect  
to  $\varphi$.   
 Moreover,   the  diagram    
 \begin{eqnarray}\label{diag-mgoan1}
 \xymatrix{
 R_*(\mathcal{K}_i,\alpha,m)\ar[rr]^-{\beta_*} \ar[d] &&  R_*(\mathcal{K}_i,\alpha,m-2s)  \ar[d] 
 \\
  R_*(\mathcal{K}'_i,{\rm  Ext}_*(\varphi)(\alpha),m)\ar[rr]^-{ ({\rm  Ext}_*(\varphi)(\beta) )_*} &&  R_*(\mathcal{K}'_i,{\rm  Ext}_*(\varphi)(\alpha),m-2s)   
 }
 \end{eqnarray}
    commutes  for  $i=1,2$.   Letting  $\beta$  be  $\beta_1$,  $\beta_2$  and  $\beta_1\wedge \beta_2$  respectively  and   applying   the  homology  functor  to the  chain  complexes  and  the  chain  maps  in  (\ref{diag-mgoan1}),    
 we  have  that  the   commutative  diagram        (\ref{eq-diag-k}) is   functorial   with respect  
to    $\varphi$.

(2).   Let  $\mathcal{L}_1$  and  $\mathcal{L}_2$  be  independence  hypergraphs  with  vertices  from  $V$  
 and  let  $\mathcal{L}'_1$  and  $\mathcal{L}'_2$  be  independence  hypergraphs  with  vertices  from  $V'$.   
 Suppose   $\varphi_1:  \mathcal{L}_1\longrightarrow  \mathcal{L}'_1$  and  $\varphi_2:  \mathcal{L}_2\longrightarrow  \mathcal{L}'_2$   are    morphisms  of  independence  hypergraphs   induced  by  the   same  bijective  map     $\varphi:  V\longrightarrow  V'$   between  the  vertices.   
  Then  we  have  induced  morphisms  of  independence  hypergraphs  
 \begin{eqnarray*}
 \varphi_1\cap\varphi_2: &&  \mathcal{L}_1\cap  \mathcal{L}_2 \longrightarrow  \mathcal{L}'_1\cap  \mathcal{L}'_2,\\
  \varphi_1\cup\varphi_2: && \mathcal{L}_1\cup  \mathcal{L}_2 \longrightarrow  \mathcal{L}'_1\cup  \mathcal{L}'_2
 \end{eqnarray*}
 which  are  also  induced  by  $\varphi$.    
 We  have    co-chain  maps    
 \begin{eqnarray}\label{eq-vmwf1}
 \varphi_{\#}: ~~~ R^*(\mathcal{L}_i,\omega,m)\longrightarrow   R^*(\mathcal{L}'_i,{\rm  Ext}^*(\varphi)(\omega),m)
 \end{eqnarray}
 for  $i=1,2$  as  well  as   
    co-chain  maps  
  \begin{eqnarray}\label{eq-vmwf2}
 \varphi_{\#}: &&  R^*(\mathcal{L}^1\cap\mathcal{L}^2,\alpha,m)\longrightarrow   R_*(\mathcal{L}'^1\cap \mathcal{L}'^2,{\rm  Ext}^*(\varphi)(\omega),m),\\
 \label{eq-vmwf3}
  \varphi_{\#}: &&  R^*(\mathcal{L}^1\cup\mathcal{L}^2,\alpha,m)\longrightarrow   R_*(\mathcal{L}'^1\cup \mathcal{L}'^2,{\rm  Ext}^*(\varphi)(\omega),m).  
 \end{eqnarray}
Applying   the  cohomology  functor  to the  co-chain  complexes  and  the co-chain  maps  in  (\ref{eq-vmwf1}),   (\ref{eq-vmwf2})   and  (\ref{eq-vmwf3}),     we  have  that  the   long  exact   sequence  (\ref{eq-lex2b})      is   functorial   with respect  
to  $\varphi$.   
 Moreover,   the  diagram    
 \begin{eqnarray}\label{oadnfmao}
 \xymatrix{
 R^*(\mathcal{L}_i,\omega,m)\ar[rr]^-{\mu_*} \ar[d] &&  R^*(\mathcal{L}_i,\omega,m+2s)  \ar[d] 
 \\
  R^*(\mathcal{L}'_i,{\rm  Ext}^*(\varphi)(\omega),m)\ar[rr]^-{ ({\rm  Ext}^*(\varphi)(\mu) )_*} &&  R_*(\mathcal{L}'_i,{\rm  Ext}_*(\varphi)(\omega),m+2s)   
 }
 \end{eqnarray}
 commutes   for  $i=1,2$.   Letting  $\mu$  be  $\mu_1$,  $\mu_2$  and  $\mu_1\wedge \mu_2$  respectively
  and  applying  the  cohomology  functor  to  the  co-chain complexes  and  the  co-chain  maps  in  (\ref{oadnfmao}),    
 we  have  that  the   commutative  diagram        (\ref{eq-diag-l}) is   functorial   with respect  
to    $\varphi$.  
 \end{proof}
 
 Let  ${\bf  MV}_*(\mathcal{K}_1,\mathcal{K}_2,\alpha,m)$  and   $ {\bf  MV}_*(\mathcal{K}'_1,\mathcal{K}'_2,\alpha',m')$   be  two  Mayer-Vietoris  sequences  of  constrained  homology  of  simplicial complexes.  
  A  morphism       
  \begin{eqnarray*}
 \Phi_*:  ~~~  {\bf  MV}_*(\mathcal{K}_1,\mathcal{K}_2,\alpha,m)\longrightarrow  {\bf  MV}_*(\mathcal{K}'_1,\mathcal{K}'_2,\alpha',m') 
  \end{eqnarray*}
     is  a  sequence  of  homomorphisms  of    $R$-modules,  denoted   as  vertical arrows,   such that the diagram  commutes  
 \begin{eqnarray*}
\xymatrix{
\cdots \ar[r]  &  H_n(\mathcal{K}_1 \cap  \mathcal{K}_2,\alpha, m)\ar[r]\ar[d] 
& H_n (\mathcal{K}_1,\alpha, m) \oplus   H_n( \mathcal{K}_2,\alpha, m)\ar[r]\ar[d] &
\\
\cdots\ar[r]  &  H_n(\mathcal{K}'_1 \cap  \mathcal{K}'_2,\alpha', m')\ar[r] 
& H_n(\mathcal{K}'_1,\alpha', m') \oplus   H_n( \mathcal{K}'_2,\alpha', m')\ar[r] &
}\\
\xymatrix{
\ar[r]
&  H_n(\mathcal{K}_1 \cup  \mathcal{K}_2,\alpha, m)\ar[r] \ar[d] & H_{n-1}(\mathcal{K}_1 \cap  \mathcal{K}_2,\alpha, m)\ar[r]\ar[d]  &\cdots\\
 \ar[r] 
&  H_n(\mathcal{K}'_1 \cup  \mathcal{K}'_2,\alpha', m')\ar[r]  & H_{n-1}(\mathcal{K}'_1 \cap  \mathcal{K}'_2,\alpha', m')\ar[r]  &\cdots  
}
 \end{eqnarray*}
  Similarly,   let  ${\bf  MV}^*(\mathcal{L}_1,\mathcal{L}_2,\omega,m)$  and   $ {\bf  MV}^*(\mathcal{L}'_1,\mathcal{L}'_2,\omega',m')$   be  two  Mayer-Vietoris  sequences  of  constrained  cohomology  of  independence  hypergraphs.     
  A  morphism       
  \begin{eqnarray*}
 \Phi^*:  ~~~  {\bf  MV}^*(\mathcal{L}_1,\mathcal{L}_2,\omega,m)\longrightarrow  {\bf  MV}^*(\mathcal{L}'_1,\mathcal{L}'_2,\omega',m') 
  \end{eqnarray*}
     is  a  sequence  of  homomorphisms  of    $R$-modules,  denoted  as  vertical arrows,  such that the diagram  commutes  
 \begin{eqnarray*}
\xymatrix{
\cdots \ar[r]  &  H^n(\mathcal{L}_1 \cap  \mathcal{L}_2,\omega, m)\ar[r]\ar[d] 
& H^n (\mathcal{L}_1,\omega, m) \oplus   H^n( \mathcal{L}_2,\omega, m)\ar[r]\ar[d] &
\\
\cdots\ar[r]  &  H^n(\mathcal{L}'_1 \cap  \mathcal{L}'_2,\omega', m')\ar[r] 
& H^n(\mathcal{L}'_1,\omega', m') \oplus   H^n( \mathcal{L}'_2,\omega', m')\ar[r] &
}\\
\xymatrix{
\ar[r]
&  H^n(\mathcal{L}_1 \cup  \mathcal{L}_2,\omega, m)\ar[r] \ar[d] & H^{n+1}(\mathcal{L}_1 \cap  \mathcal{L}_2,\omega, m)\ar[r]\ar[d]  &\cdots\\
 \ar[r] 
&  H^n(\mathcal{L}'_1 \cup  \mathcal{L}'_2,\omega', m')\ar[r]  & H^{n+1}(\mathcal{L}'_1 \cap  \mathcal{L}'_2,\omega', m')\ar[r]  &\cdots 
}
 \end{eqnarray*} 
The  next  proposition  is  a  slight  generalization   of   \cite[Proposition~4.1]{mv}.  
\begin{proposition}
\label{pr-3.9nm21}
For  any     
simplicial  complexes  $\mathcal{K}_1$  and  $\mathcal{K}_2$   and  any    
 independence  hypergraphs   $\mathcal{L}_1$  and  $\mathcal{L}_2$      with  their  vertices from  $V$,
\begin{enumerate}[(1).]
\item
we  have  an  induced  morphism  
\begin{eqnarray}\label{eq-5.1.bjkl1}
\beta_*: ~~~ {\bf  MV}_*(\mathcal{K}_1,\mathcal{K}_2,\alpha,m)\longrightarrow  {\bf  MV}_*(\mathcal{K}_1,\mathcal{K}_2,\alpha,m-2s). 
\end{eqnarray}
 Moreover,   $(\beta_1\wedge\beta_2)_*= (\beta_1)_* (\beta_2)_*= (\beta_2)_* (\beta_1)_*$,     i.e.  the  diagram  commutes 
 \begin{eqnarray}\label{eq-5.1vjwfg1}
 \xymatrix{
 {\bf  MV}_*(\mathcal{K}_1,\mathcal{K}_2,\alpha,m) \ar[r]^-{(\beta_1)_*} \ar[rdd]^-{(\beta_1\wedge \beta_2)_*}
 \ar[dd]_-{(\beta_2)_*}  &{\bf  MV}_*(\mathcal{K}_1,\mathcal{K}_2,\alpha,m-2s_1)\ar[dd]^-{(\beta_2)_*}\\
 \\
 {\bf  MV}_*(\mathcal{K}_1,\mathcal{K}_2,\alpha,m-2s_2)\ar[r]^-{(\beta_1)_*}& {\bf  MV}_*(\mathcal{K}_1,\mathcal{K}_2,\alpha,m-2s_1-2s_2).  
 }
 \end{eqnarray}
Furthermore,   the  diagram  is  functorial  with  respect to  simplicial  maps   induced  by  bijective  maps  between  the  vertices; 

\item
we  have  an  induced  morphism  
\begin{eqnarray}\label{eq-5.1.bjkl2}
\mu_*: ~~~ {\bf  MV}^*(\mathcal{L}_1,\mathcal{L}_2,\omega,m)\longrightarrow  {\bf  MV}^*(\mathcal{L}_1,\mathcal{L}_2,\omega,m+2s). 
\end{eqnarray}
Moreover,   $(\mu_1\wedge\mu_2)_*= (\mu_1)_* (\mu_2)_*= (\mu_2)_* (\mu_1)_*$,     i.e.  the  diagram  commutes 
 \begin{eqnarray}\label{eq-5.1vjwfg2}
 \xymatrix{
 {\bf  MV}^*(\mathcal{L}_1,\mathcal{L}_2,\omega,m) \ar[r]^-{(\mu_1)_*} \ar[rdd]^-{(\mu_1\wedge \mu_2)_*}
 \ar[dd]_-{(\mu_2)_*}  &{\bf  MV}^*(\mathcal{L}_1,\mathcal{L}_2,\omega,m+2s_1)\ar[dd]^-{(\mu_2)_*}\\
 \\
 {\bf  MV}^*(\mathcal{L}_1,\mathcal{L}_2,\omega,m+2s_2)\ar[r]^-{(\mu_1)_*}& {\bf  MV}^*(\mathcal{L}_1,\mathcal{L}_2,\omega,m+2s_1+2s_2).  
 }
 \end{eqnarray}
 Furthermore,   the  diagram  is  functorial  with  respect to  morphisms  of  independence  hypergraphs 
 induced  by  bijective  maps  between  the  vertices.   
\end{enumerate}
\end{proposition}

\begin{proof}
By  a  similar argument  of  \cite[Proposition~4.1]{mv}~(1)  and  (2),  we  will  obtain  (1)  and  (2)  respectively.  
\end{proof}

 \subsection{Constrained  (co)homology  for  general  hypergraphs}

 We   investigate  the  constrained  homology of  the  (lower-)associated  simplicial  complex  
   and  the  constrained  cohomology  of  the  (lower-)associated  independence  hypergraph   for     a   general  hypergraph.    
 Let  $\mathcal{H}$  be  a  hypergraph  with  its   vertices  from  $V$.  The {\it associated  simplicial complex}  of  $\mathcal{H}$  is  the  smallest  simplicial complex  containing   $\mathcal{H}$   explicitly  given by  (cf.    \cite{hg1,parks, jktr1000,jktr2,stability11111,jktr3})
   \begin{eqnarray*}
 \Delta\mathcal{H}=\{\tau\in\Delta[V]\mid \tau\subseteq \sigma  {\rm~for~some~} \sigma\in \mathcal{H}\}.
 \end{eqnarray*}
 The {\it  lower-associated  simplicial  complex}   of  $\mathcal{H}$   is   the  largest  simplicial complex  contained  in  $\mathcal{H}$  explicitly  given  by  (cf.     \cite{jktr1000,jktr2,stability11111, jktr3})
   \begin{eqnarray*}
 \delta\mathcal{H}=\{\sigma\in\mathcal{H} \mid  \tau\in\mathcal{H}  {\rm ~for~any~}\tau\subseteq  \sigma {\rm~and~}  \tau\neq\emptyset\}. 
 \end{eqnarray*}
  The  {\it  associated independence  hypergraph}  of  $\mathcal{H}$   is   the   smallest  independence hypergraph  containing   $\mathcal{H}$  explicitly  given  by   (cf.  \cite{grh1})
 \begin{eqnarray*}
 \bar  \Delta\mathcal{H}  = \{\tau\in\Delta[V]\mid \tau\supseteq \sigma {\rm~for~some~} \sigma\in \mathcal{H}\}.  
  \end{eqnarray*}
  The  {\it lower-associated independence  hypergraph}  of  $\mathcal{H}$   is   the   largest   independence hypergraph  explicitly  given  by   (cf.  \cite{grh1})
  \begin{eqnarray*}
  \bar\delta \mathcal{H} =\{\sigma\in\mathcal{H}\mid  \tau\in\mathcal{H}  {\rm ~for~any~}\tau\supseteq  \sigma{\rm~and~}    \tau\in  \Delta[  V]\}.
  \end{eqnarray*}

  \begin{lemma}\label{le-5.a}
  For  any  hypergraph  $\mathcal{H}$,  
  \begin{enumerate}[(1).]
  \item
  The  canonical  inclusion  $\iota:   \delta\mathcal{H}\longrightarrow  \Delta\mathcal{H}$
   of  simplicial  complexes   
  induces  a  chain  map  
  \begin{eqnarray*}
   \iota_\#: ~~~ R_*(\delta\mathcal{H},\alpha,m)\longrightarrow  R_*(\Delta\mathcal{H},\alpha,m);      
  \end{eqnarray*}
  \item
  The canonical  inclusion  
  $\bar\iota:  \bar\delta\mathcal{H}\longrightarrow \bar\Delta\mathcal{H}$  
  of  independence  hypergraphs  induces  a  co-chain  map 
    \begin{eqnarray*}
  \bar\iota_\#:  ~~~ R^*(\bar\delta\mathcal{H},\omega,m)\longrightarrow  R^*(\bar\Delta\mathcal{H},\omega,m).  
  \end{eqnarray*}
\end{enumerate}
  \end{lemma}
  \begin{proof}
  The  canonical  inclusions  $\iota$  and  $\bar\iota$   are  both   induced  by the  identity  map  ${\rm  id}$  on  $V$.  
  The  identity  map   ${\rm  id}$  on  $V$  induces  the  identity  map  ${\rm  Ext}_*({\rm  id})$  on  ${\rm  Ext}_*(V)$   
     which  sends  $\alpha$  to  itself  and  the  identity  map   ${\rm  Ext}^*({\rm  id})$  on  ${\rm  Ext}^*(V)$   
     which  sends  $\omega$  to  itself.  
  By  (\ref{eq-vb345}),   the  canonical  inclusion  $\iota:   \delta\mathcal{H}\longrightarrow  \Delta\mathcal{H}$ 
 induces  a  commutative  diagram    
 \begin{eqnarray*}
\xymatrix{
\cdots\ar[r]^-{\alpha} 
& R_{(n+\lambda)(2t+1)+q}(\delta\mathcal{H})\ar[r]^-{\alpha}\ar[d] 
 &R_{(n-1+\lambda)(2t+1)+q}(\delta\mathcal{H})\ar[r]^-{\alpha}\ar[d] &
 \\
 \cdots\ar[r]^-{\alpha} 
& R_{(n+\lambda)(2t+1)+q}(\Delta\mathcal{H})\ar[r]^-{\alpha} 
 &R_{(n-1+\lambda)(2t+1)+q}(\Delta\mathcal{H})\ar[r]^-{\alpha} &
 }\\
\xymatrix{
  \cdots\ar[r] ^-{\alpha} 
 &R_{(1+\lambda)(2t+1)+q}(\delta\mathcal{H})\ar[r]^-{\alpha} \ar[d]
 &R_{\lambda(2t+1)+q}(\delta\mathcal{H})\ar[r]^-{\alpha} \ar[d]
 &0\\
 \cdots\ar[r]^-{\alpha} 
 &R_{(1+\lambda)(2t+1)+q}(\Delta\mathcal{H})\ar[r]^-{\alpha} 
 &R_{\lambda(2t+1)+q}(\Delta\mathcal{H})\ar[r]^-{\alpha} 
 &0,
 }
\end{eqnarray*}
where   each  row  is  
  a chain complex.  
  The  vertical  maps  give  a  chain  map  
$  \iota_\#$.   We  obtain  (1).   
     By  (\ref{eq-vb368}),   the  canonical  inclusion  $\bar\iota:   \bar\delta\mathcal{H}\longrightarrow  \bar\Delta\mathcal{H}$ 
 induces  a  commutative  diagram    
 \begin{eqnarray*}
\xymatrix{
\cdots 
& R_{(n+\lambda)(2t+1)+q}(\bar\delta\mathcal{H})\ar[d] \ar[l]_-{\omega} 
 &R_{(n-1+\lambda)(2t+1)+q}(\bar\delta\mathcal{H}) \ar[d]\ar[l]_-{\omega}  & \ar[l]_-{\omega} 
 \\
 \cdots 
& R_{(n+\lambda)(2t+1)+q}(\bar\Delta\mathcal{H})\ar[l]_-{\omega} 
 &R_{(n-1+\lambda)(2t+1)+q}(\bar\Delta\mathcal{H})\ar[l]_-{\omega} &\ar[l]_-{\omega} 
}\\
\xymatrix{
  \cdots 
 &R_{(1+\lambda)(2t+1)+q}(\bar\delta\mathcal{H})\ar[l]_-{\omega} \ar[d]
 &R_{\lambda(2t+1)+q}(\bar\delta\mathcal{H})\ar[l]_-{\omega} \ar[d]
 &0\ar[l]_-{\omega} \\
 \cdots
 &R_{(1+\lambda)(2t+1)+q}(\bar\Delta\mathcal{H})\ar[l]_-{\omega} 
 &R_{\lambda(2t+1)+q}(\bar\Delta\mathcal{H})\ar[l]_-{\omega} 
 &0\ar[l]_-{\omega} ,
 }
\end{eqnarray*}
where   each  row  is  
  a  co-chain   complex.  
  The  vertical  maps  give  a  co-chain  map  
  $
  \bar\iota_\#$.   We   obtain      (2).  
  \end{proof}
  
  \begin{proposition}\label{pr-5.mmm1}
  For  any     hypergraphs   $\mathcal{H}_1$  and  $\mathcal{H}_2$      with  vertices  from  $V$,  
  \begin{enumerate}[(1).]
  \item
   we  have  a  commutative  diagram  
   \begin{eqnarray*}
   \xymatrix{
   \cdots \ar[r]   &    H_n(\delta\mathcal{H}_1 \cap  \delta\mathcal{H}_2,\alpha, m) \ar[r] \ar[d]
 & H_n (\delta\mathcal{H}_1,\alpha, m) \oplus   H_n(\delta\mathcal{H}_2,\alpha, m)   \ar[d]\ar[r] &
 \\
\cdots \ar[r]   &    H_n(\Delta\mathcal{H}_1 \cap  \Delta\mathcal{H}_2,\alpha, m)  \ar[r] 
 & H_n (\Delta\mathcal{H}_1,\alpha, m) \oplus   H_n(\Delta\mathcal{H}_2,\alpha, m)\ar[r] &   
 }\\
\xymatrix{
 \ar[r]   
&   H_n(\delta\mathcal{H}_1 \cup  \delta\mathcal{H}_2,\alpha, m)\ar[r] \ar[d]  &H_{n-1}(\delta\mathcal{H}_1 \cap  \delta\mathcal{H}_2,\alpha, m)\ar[r] \ar[d]  &\cdots\\
\ar[r]  
&   H_n(\Delta\mathcal{H}_1 \cup  \Delta\mathcal{H}_2,\alpha, m)\ar[r] &H_{n-1}(\Delta\mathcal{H}_1 \cap  \Delta\mathcal{H}_2,\alpha, m)\ar[r] &\cdots  
 }  
   \end{eqnarray*}
   where  both  rows  are     long  exact  sequences  of  constrained  homology  groups  and  the  vertical  maps  are  induced  by  the  canonical  inclusions  from  $\delta\mathcal{H}_1$  into  $\Delta\mathcal{H}_1$  and  from  $\delta\mathcal{H}_2$  into  $\Delta\mathcal{H}_2$;  
   \item
   we  have  a  commutative  diagram  
   \begin{eqnarray*}
   \xymatrix{
   \cdots \ar[r]   &    H^n(\bar\delta\mathcal{H}_1 \cap \bar \delta\mathcal{H}_2,\omega, m) \ar[r] \ar[d]
 & H^n (\bar\delta\mathcal{H}_1,\omega, m) \oplus   H^n(\bar\delta\mathcal{H}_2,\omega, m)   \ar[d] \ar[r]&
 \\
\cdots \ar[r]   &    H^n(\bar\Delta\mathcal{H}_1 \cap  \bar\Delta\mathcal{H}_2,\omega, m)  \ar[r] 
 & H^n (\bar\Delta\mathcal{H}_1,\omega, m) \oplus   H^n(\bar\Delta\mathcal{H}_2,\omega, m)   \ar[r] &
 }\\
\xymatrix{
 \ar[r]   
&   H^n(\bar\delta\mathcal{H}_1 \cup  \bar\delta\mathcal{H}_2,\omega, m)\ar[r] \ar[d]  &H^{n+1}(\bar\delta\mathcal{H}_1 \cap  \bar\delta\mathcal{H}_2,\omega, m)\ar[r] \ar[d]  &\cdots\\
\ar[r]  
&   H^n(\bar\Delta\mathcal{H}_1 \cup  \bar\Delta\mathcal{H}_2,\omega, m)\ar[r] &H^{n+1}(\bar\Delta\mathcal{H}_1 \cap  \bar\Delta\mathcal{H}_2,\omega, m)\ar[r] &\cdots  
 }  
   \end{eqnarray*}
   where  both  rows  are     long  exact  sequences  of  constrained cohomology groups  and  the  vertical  maps  are  induced  by  the  canonical  inclusions  from  $\bar\delta\mathcal{H}_1$  into  $\bar\Delta\mathcal{H}_1$  and  from  $\bar\delta\mathcal{H}_2$  into  $\bar\Delta\mathcal{H}_2$. 
   \end{enumerate}
  \end{proposition}
  
  \begin{proof}
  (1).   Let  $\mathcal{K}$  be  $\delta\mathcal{H}$  and  $\Delta\mathcal{H}$  respectively  in   (\ref{eq-lex1a}). 
  We  obtain     the  two  rows  in the  diagram  in  (1)   which   are  long  exact  sequences.  
    Let  $\iota_i:  \delta\mathcal{H}_i\longrightarrow  \Delta\mathcal{H}_i$   be  the  canonical  
    inclusions  of  the  lower-associated  simplicial  complex  of  $\mathcal{H}_i$   into  the   associated  simplicial  complex  of  $\mathcal{H}_i$,  $i=1,2$.   
    Both  $\iota_1$  and  $\iota_2$  are   induced  by  the  identity  map  on  $V$.  
      By  Proposition~\ref{pr-5.a1}~(1),   the  long  exact sequence     (\ref{eq-lex1a})  is  functorial  with  respect to   $\iota_1$,  $\iota_2$,  $\iota_1\cap\iota_2$  and  $\iota_1\cup\iota_2$.     Precisely,  the  chain  maps  
     \begin{eqnarray*}
       (\iota_i)_\#:  ~~~ R_*(\delta\mathcal{H}_i,\alpha,m)\longrightarrow  R_*(\Delta\mathcal{H}_i,\alpha,m),   
       \end{eqnarray*}
     $i=1,2$,   together  with  the  induced  chain  maps 
     \begin{eqnarray*}
       (\iota_1\cap\iota_2)_\#:  && R_*(\delta\mathcal{H}_1\cap\delta\mathcal{H}_2, \alpha,m )\longrightarrow  
       R_*(\Delta\mathcal{H}_1\cap\Delta\mathcal{H}_2, \alpha,m ),\\
       (\iota_1\cup\iota_2)_\#:  && R_*(\delta\mathcal{H}_1\cup\delta\mathcal{H}_2, \alpha,m )\longrightarrow  
       R_*(\Delta\mathcal{H}_1\cup\Delta\mathcal{H}_2, \alpha,m )
       \end{eqnarray*}
induce  the  vertical  maps   
    between  the  constrained  homology  groups  such  that the  diagram  in  (1)  commutes.

  (2).  Let  $\mathcal{L}$  be  $\bar\delta\mathcal{H}$  and  $\bar\Delta\mathcal{H}$  respectively  in   (\ref{eq-lex2b}). 
  We  obtain    the  two  rows  in the  diagram   in  (2)   which  are  long  exact  sequences.  
  Let  $\bar\iota_i:  \bar\delta\mathcal{H}_i\longrightarrow  \bar\Delta\mathcal{H}_i$   be  the  canonical  
    inclusions  of  the  lower-associated  independence  hypergraph   of  $\mathcal{H}_i$   into  the   associated   independence  hypergraph    of  $\mathcal{H}_i$,  $i=1,2$.   
    Both  $\bar\iota_1$  and  $\bar\iota_2$  are   induced  by  the  identity  map  on  $V$.  
      By  Proposition~\ref{pr-5.a1}~(2),   the  long  exact sequence     (\ref{eq-lex2b})  is  functorial  with  respect to   $\bar\iota_1$,  $\bar\iota_2$,  $\bar\iota_1\cap\bar\iota_2$  and  $\bar\iota_1\cup\bar\iota_2$.     Precisely,  the  co-chain  maps  
     \begin{eqnarray*}
       (\bar\iota_i)_\#:  ~~~ R^*(\bar\delta\mathcal{H}_i,\omega,m)\longrightarrow  R^*(\bar\Delta\mathcal{H}_i,\omega,m),   
       \end{eqnarray*}
     $i=1,2$,   together  with  the  induced  co-chain  maps 
     \begin{eqnarray*}
       (\bar\iota_1\cap\bar\iota_2)_\#:  && R^*(\bar\delta\mathcal{H}_1\cap\bar\delta\mathcal{H}_2, \omega,m )\longrightarrow  
       R^*(\bar\Delta\mathcal{H}_1\cap\bar\Delta\mathcal{H}_2, \omega,m ),\\
       (\bar\iota_1\cup\bar\iota_2)_\#:  && R^*(\bar\delta\mathcal{H}_1\cup\bar\delta\mathcal{H}_2, \omega,m )\longrightarrow  
       R^*(\bar\Delta\mathcal{H}_1\cup\bar\Delta\mathcal{H}_2, \omega,m )
       \end{eqnarray*}
induce  the  vertical  maps   
    between  the  constrained  cohomology  groups  such  that the  diagram  in  (2)  commutes.  
  \end{proof}
  
  \begin{definition}
  \begin{enumerate}[(1).]
  \item
  We  call  the  commutative  diagram  in  Proposition~\ref{pr-5.mmm1}~(1)  the  {\it   Mayer-Vietoris  sequence}  for  the  constrained  homology  of  the    hypergraph  pairs  $(\mathcal{H}_1, \mathcal{H}_2)$   with  respect to  $\alpha$  and  $m$.   We   denote  it  as    ${\bf   MV}_*(\mathcal{H}_1, \mathcal{H}_2, \alpha,m)$;   
  \item
    We  call  the  commutative  diagram  in  Proposition~\ref{pr-5.mmm1}~(2)  the  {\it   Mayer-Vietoris  sequence}  for  the  constrained  cohomology  of  the    hypergraph  pairs  $(\mathcal{H}_1, \mathcal{H}_2)$   with  respect to  $\omega$  and  $m$.   We   denote  it  as    ${\bf   MV}^*(\mathcal{H}_1, \mathcal{H}_2, \omega,m)$.  
  \end{enumerate}
  \end{definition}
  
\begin{remark}
Note  that 
${\bf   MV}_*(\mathcal{H}_1, \mathcal{H}_2, \alpha,m)$  consists  of  two  Mayer-Vietoris    sequences  
  of  simplicial  complexes   
\begin{eqnarray*}
  {\bf   MV}_*(\delta\mathcal{H}_1, \delta\mathcal{H}_2, \alpha,m),  
~~~~~~   {\bf   MV}_*(\Delta\mathcal{H}_1, \Delta\mathcal{H}_2, \alpha,m) 
\end{eqnarray*}
together  with  a  morphism 
\begin{eqnarray*}
 (\iota_1,\iota_2)_*:~~~ {\bf   MV}_*(\delta\mathcal{H}_1, \delta\mathcal{H}_2, \alpha,m)\longrightarrow  {\bf   MV}_*(\Delta\mathcal{H}_1, \Delta\mathcal{H}_2, \alpha,m)   
\end{eqnarray*}
and  ${\bf   MV}^*(\mathcal{H}_1, \mathcal{H}_2, \omega,m)$  consists  of  two  Mayer-Vietoris    sequences  
  of  independence  hypergraphs   
\begin{eqnarray*}
  {\bf   MV}^*(\bar\delta\mathcal{H}_1, \bar\delta\mathcal{H}_2, \omega,m),  
~~~~~~   {\bf   MV}^*(\bar\Delta\mathcal{H}_1, \bar\Delta\mathcal{H}_2, \omega,m) 
\end{eqnarray*}
together  with  a  morphism 
\begin{eqnarray*}
 (\bar\iota_1,\bar\iota_2)^*:~~~ {\bf   MV}^*(\bar\delta\mathcal{H}_1, \bar\delta\mathcal{H}_2, \omega,m)\longrightarrow  {\bf   MV}^*(\bar\Delta\mathcal{H}_1, \bar\Delta\mathcal{H}_2, \omega,m).     
\end{eqnarray*}
\end{remark}

Let  ${\bf  MV}_*(\mathcal{H}_1,\mathcal{H}_2,\alpha,m)$  and   $ {\bf  MV}_*(\mathcal{H}'_1,\mathcal{H}'_2,\alpha',m')$   be  two  Mayer-Vietoris  sequences  of  constrained  homology  of  hypergraphs.  
  A  {\it  morphism}       
  \begin{eqnarray}\label{eq-5.98aq1}
 \Phi_*:  ~~~  {\bf  MV}_*(\mathcal{H}_1,\mathcal{H}_2,\alpha,m)\longrightarrow  {\bf  MV}_*(\mathcal{H}'_1,\mathcal{H}'_2,\alpha',m') 
  \end{eqnarray}
     is   a   commutative  diagram 
     \begin{eqnarray}\label{eq-5.98aq2}
     \xymatrix{
     {\bf  MV}_*(\delta\mathcal{H}_1,\delta\mathcal{H}_2,\alpha,m)\ar[r]^-{\delta\Phi_*}\ar[d]  &      {\bf  MV}_*(\delta\mathcal{H}'_1,\delta\mathcal{H}'_2,\alpha',m')\ar[d]\\
          {\bf  MV}_*(\Delta\mathcal{H}_1,\Delta\mathcal{H}_2,\alpha,m)\ar[r]^-{\Delta\Phi_*}  &      {\bf  MV}_*(\Delta\mathcal{H}'_1,\Delta\mathcal{H}'_2,\alpha',m') 
     }
     \end{eqnarray} 
 where  the  vertical  maps  are  induced  by  the  canonical  inclusions  of  $\delta\mathcal{H}_i$  into  $\Delta\mathcal{H}_i$,  $i=1,2$.   
  Similarly,   let  ${\bf  MV}^*(\mathcal{H}_1,\mathcal{H}_2,\omega,m)$  and   $ {\bf  MV}^*(\mathcal{H}'_1,\mathcal{H}'_2,\omega',m')$   be  two  Mayer-Vietoris  sequences  of  constrained  cohomology  of  hypergraphs.  
  A  {\it  morphism}       
  \begin{eqnarray}\label{eq-5.98aq3}
 \Phi^*:  ~~~  {\bf  MV}^*(\mathcal{H}_1,\mathcal{H}_2,\omega,m)\longrightarrow  {\bf  MV}^*(\mathcal{H}'_1,\mathcal{H}'_2,\omega',m') 
  \end{eqnarray}
     is   a   commutative  diagram 
     \begin{eqnarray}\label{eq-5.98aq5}
     \xymatrix{
     {\bf  MV}^*(\bar\delta\mathcal{H}_1,\bar\delta\mathcal{H}_2,\omega,m)\ar[r]^-{\bar\delta\Phi^*}\ar[d]  &      {\bf  MV}^*(\bar\delta\mathcal{H}'_1,\bar\delta\mathcal{H}'_2,\omega',m')\ar[d]\\
          {\bf  MV}^*(\bar\Delta\mathcal{H}_1,\bar\Delta\mathcal{H}_2,\omega,m)\ar[r]^-{\bar\Delta\Phi^*}  &      {\bf  MV}^*(\bar\Delta\mathcal{H}'_1,\bar\Delta\mathcal{H}'_2,\omega',m') 
     }
     \end{eqnarray} 
 where  the  vertical  maps  are  induced  by  the  canonical  inclusions  of    $\bar\delta\mathcal{H}_i$  into  $\bar\Delta\mathcal{H}_i$,  $i=1,2$.

  Let  $\mathcal{H}$  and  $\mathcal{H}'$  be   hypergraphs  with  vertices  from  $V$  and  $V'$ respectively.  
  Let  $\varphi:  \mathcal{H}\longrightarrow \mathcal{H}'$  be  a  morphism  of  hypergraphs.  
  Then  we  have  a  commutative diagram  
  \begin{eqnarray}\label{eq-vmbg1}
  \xymatrix{
  \delta\mathcal{H} \ar[r]^-{\delta\varphi} \ar[d]   & \delta\mathcal{H}'  \ar[d] \\
  \mathcal{H}\ar[r]^-{\varphi} \ar[d] &\mathcal{H}' \ar[d] \\
    \Delta\mathcal{H} \ar[r]^-{\Delta\varphi}    & \Delta\mathcal{H}'  
  }
  \end{eqnarray}
  where both  $\delta\varphi$  and  $\Delta\varphi$  are  simplicial  maps   and   the   vertical  maps  are  canonical  inclusions  of  hypergraphs  induced  by  the  identity  map   on  the  
  vertices.  If  $\varphi$  is  induced  by  a  bijective  map  $\varphi:  V\longrightarrow  V'$,  then   
  we  have  a  commutative diagram  
  \begin{eqnarray}\label{eq-vmbg2}
  \xymatrix{
  \bar\delta\mathcal{H} \ar[r]^-{\bar\delta\varphi} \ar[d]   & \bar\delta\mathcal{H}'  \ar[d] \\
  \mathcal{H}\ar[r]^-{\varphi} \ar[d] &\mathcal{H}' \ar[d] \\
    \bar\Delta\mathcal{H} \ar[r]^-{\bar\Delta\varphi}    & \bar\Delta\mathcal{H}'  
  }
  \end{eqnarray}
  where   both  $\bar\delta\varphi$  and  $\bar\Delta\varphi$  are  morphisms  of  independence  hypergraphs  and    the   vertical  maps  are  canonical  inclusions  of  hypergraphs  induced  by  the  identity  map   on  the  
  vertices. 
  \begin{theorem}[Main  Result  II]
  \label{pr-5.829a}
  For  any     hypergraphs   $\mathcal{H}_1$  and  $\mathcal{H}_2$      with  vertices  from  $V$,  
    \begin{enumerate}[(1).]
    \item
    we  have  an  induced  morphism  
\begin{eqnarray}\label{eq-5.2vla1}
\beta_*: ~~~ {\bf  MV}_*(\mathcal{H}_1,\mathcal{H}_2,\alpha,m)\longrightarrow  {\bf  MV}_*(\mathcal{H}_1,\mathcal{H}_2,\alpha,m-2s). 
\end{eqnarray}
 Moreover,   $(\beta_1\wedge\beta_2)_*= (\beta_1)_* (\beta_2)_*= (\beta_2)_* (\beta_1)_*$,     i.e.  the  diagram  commutes 
 \begin{eqnarray}\label{eq-5.2vlb1}
 \xymatrix{
 {\bf  MV}_*(\mathcal{H}_1,\mathcal{H}_2,\alpha,m) \ar[r]^-{(\beta_1)_*} \ar[rdd]^-{(\beta_1\wedge \beta_2)_*}
 \ar[dd]_-{(\beta_2)_*}  &{\bf  MV}_*(\mathcal{H}_1,\mathcal{H}_2,\alpha,m-2s_1)\ar[dd]^-{(\beta_2)_*}\\
 \\
 {\bf  MV}_*(\mathcal{H}_1,\mathcal{H}_2,\alpha,m-2s_2)\ar[r]^-{(\beta_1)_*}& {\bf  MV}_*(\mathcal{H}_1,\mathcal{H}_2,\alpha,m-2s_1-2s_2).  
 }
 \end{eqnarray}
Furthermore,   the  diagram  is  functorial  with  respect to  morphisms  of  hypergraphs   induced  by  bijective  maps  between  the  vertices;

    \item
      we  have  an  induced  morphism  
\begin{eqnarray}\label{eq-5.2vla2}
\mu_*: ~~~ {\bf  MV}^*(\mathcal{H}_1,\mathcal{H}_2,\omega,m)\longrightarrow  {\bf  MV}^*(\mathcal{H}_1,\mathcal{H}_2,\omega,m+2s). 
\end{eqnarray}
Moreover,   $(\mu_1\wedge\mu_2)_*= (\mu_1)_* (\mu_2)_*= (\mu_2)_* (\mu_1)_*$,     i.e.  the  diagram  commutes 
 \begin{eqnarray}\label{eq-5.2vlb2}
 \xymatrix{
 {\bf  MV}^*(\mathcal{H}_1,\mathcal{H}_2,\omega,m) \ar[r]^-{(\mu_1)_*} \ar[rdd]^-{(\mu_1\wedge \mu_2)_*}
 \ar[dd]_-{(\mu_2)_*}  &{\bf  MV}^*(\mathcal{H}_1,\mathcal{H}_2,\omega,m+2s_1)\ar[dd]^-{(\mu_2)_*}\\
 \\
 {\bf  MV}^*(\mathcal{H}_1,\mathcal{H}_2,\omega,m+2s_2)\ar[r]^-{(\mu_1)_*}& {\bf  MV}^*(\mathcal{H}_1,\mathcal{H}_2,\omega,m+2s_1+2s_2).  
 }
 \end{eqnarray}
 Furthermore,   the  diagram  is  functorial  with  respect to  morphisms  of     hypergraphs 
 induced  by  bijective  maps  between  the  vertices.   
\end{enumerate}
  \end{theorem}
  
  \begin{proof}
 (1).   Let  $\mathcal{K}$  be  $\delta\mathcal{H}$  and  $\Delta\mathcal{H}$   respectively  
  in  Proposition~\ref{pr-3.9nm21}~(1).  
  Since  (\ref{eq-5.1.bjkl1})  as  well as   the  commutative  diagram  (\ref{eq-5.1vjwfg1})   is  functorial  
  with  respect  to  the  simplicial  maps  $\iota_i:  \delta\mathcal{H}_i\longrightarrow   \Delta\mathcal{H}_i$, 
   $i=1,2$,    we  obtain  (\ref{eq-5.2vla1})  as   well  as  the  commutative  diagram   (\ref{eq-5.2vlb1}) of  the  Mayer-Vietoris  sequences  of  constrained  homology  of  hypergraphs   in  (1).  
  By  the  commutative  diagram  (\ref{eq-vmbg1}),   any  morphism   of  hypergraphs   $\varphi: \mathcal{H}\longrightarrow \mathcal{H}'$     induces    simplicial  maps  $\delta\varphi$  and  $\Delta\varphi$  
  such  that   the  diagram   (\ref{eq-vmbg1})  commutes.  Therefore,  since  
  the  diagram  in  Proposition~\ref{pr-3.9nm21}~(1)  is  functorial  with respect to  simplicial  maps  induced  by  bijective  maps  between  the  vertices,  
  it  follows  that    
  the  diagram  in  (1)  is  functorial  with  respect to  morphisms  of  hypergraphs  induced  by  bijective  maps  between  the  vertices.

  (2).    Let  $\mathcal{L}$  be  $\bar\delta\mathcal{H}$  and  $\bar\Delta\mathcal{H}$   respectively  
  in  Proposition~\ref{pr-3.9nm21}~(2).  
  Since  (\ref{eq-5.1.bjkl2})  as  well as   the  commutative  diagram  (\ref{eq-5.1vjwfg2})      is  functorial  
  with  respect  to  the  morphisms  of   independence   hypergraphs   $\bar\iota_i:  \bar\delta\mathcal{H}_i\longrightarrow   \bar\Delta\mathcal{H}_i$, 
   $i=1,2$,    which  are  induced  by  bijective  maps  between  the  vertices,  we  obtain  (\ref{eq-5.2vla2})  as   well  as  the  commutative  diagram   (\ref{eq-5.2vlb2}) of  the  Mayer-Vietoris  sequences  of  constrained  cohomology  of  hypergraphs   in  (2).  
  By  the  commutative  diagram  (\ref{eq-vmbg2}),   any  morphism   of  hypergraphs   $\varphi: \mathcal{H}\longrightarrow \mathcal{H}'$   induced  by  a  bijective  map   between  the  vertices   will        induce  morphisms  of  independence  hypergraphs       $\bar\delta\varphi$  and  $\bar\Delta\varphi$  
  such  that   the  diagram   (\ref{eq-vmbg1})  commutes.  Therefore,  since  
  the  diagram  in  Proposition~\ref{pr-3.9nm21}~(2)  is  functorial  with respect to  morphisms  of  independence hypergraphs     induced  by  bijective  maps  between  the  vertices,  
  it  follows  that    
  the  diagram  in  (2)  is  functorial  with  respect to  morphisms  of  hypergraphs    induced  by  bijective  maps  between  the  vertices.   
  \end{proof}

 \subsection{Examples}

\begin{example}\label{ex-8.111}
Let  $w:  V\longrightarrow  R$.   Let   $\mathcal{K}$  be  a  simplicial  complex  with  vertices  from  $V$.  Let $\mathcal{L}$  be  an  independence   hypergraph  with  vertices  from  $V$.  
\begin{enumerate}[(1).]
\item
 Let $\{v_0,v_1,\ldots, v_n\}$  be an $n$-simplex of  $\mathcal{K}$.
The usual boundary operator   (cf.  \cite[p. 105]{hatcher},  \cite[p.  28]{eat})  is  
a  homomorphism  of  graded  $R$-modules 
\begin{eqnarray*}
\partial_n:  ~~~  R_n(\mathcal{K})\longrightarrow   R_{n-1}(\mathcal{K})
\end{eqnarray*}
 given by
\begin{eqnarray}\label{eq-usual-boundary}
\partial_n\{v_0,v_1,\ldots, v_n\}=\sum_{i=0}^n  (-1)^i   \{v_0,\ldots, \widehat{v_i},\ldots, v_n\}.
  \end{eqnarray}
  It  is  direct  that   $\partial_n\circ\partial_{n+1}=0$.  We   have  
\begin{eqnarray}\label{eq-6.afv1}
\partial_n=\sum_{v\in V}\frac{\partial}{\partial v}.
\end{eqnarray}
  The   usual  homology  group  of  $\mathcal{K}$  is  
  \begin{eqnarray}\label{eq-ih-a1}
  H_n(\mathcal{K};R)= {\rm Ker}(\partial_n)/{\rm Im}(\partial_{n+1}).  
  \end{eqnarray}
\item
Let  $\{v_0,v_1,\ldots, v_n\}$  be an $n$-hyperedge   of  $\mathcal{L}$.   
Define  the  coboundary  operator  to  be   a   homomorphism  of   graded  $R$-modules 
\begin{eqnarray*}
\partial^n:  ~~~  R_n(\mathcal{L})\longrightarrow   R_{n+1}(\mathcal{L})
\end{eqnarray*}
given  by 
\begin{eqnarray}\label{usual-cob}
\partial^n(\{v_0,v_1,\ldots,v_n\})=\sum_{v\in V}  {{\rm sgn}(v)} \{v_0,v_1,\ldots,v_n,v\},    
\end{eqnarray} 
where 
\begin{eqnarray*}
{\rm sgn}(v)=\begin{cases}
0,  &{\rm  ~if  ~} v\in \{v_0,v_1,\ldots,v_n\},\\
(-1)^i, &{\rm  ~if ~}  v_{i-1}\prec  v\prec v_i{\rm~for~some~} 1\leq i\leq n,\\
(-1)^{n+1}, &{\rm  ~if ~}   v_n\prec v,\\
1, &{\rm  ~if ~}   v\prec v_0.  
\end{cases}
\end{eqnarray*}
It  is  direct  to  verify  that   $\partial^{n+1}\circ\partial^{n}=0$.  
We  have  
\begin{eqnarray}\label{eq-6.afv2}
\partial^n=\sum_{v\in V}  dv.
\end{eqnarray}
  We  define  the  $n$-th upper-homology group of $\mathcal{L}$  as     the quotient $R$-module 
\begin{eqnarray}\label{eq-ih-a2}
H^n(\mathcal{L};R)={\rm Ker}(\partial^n)/{\rm Im}(\partial^{n-1}). 
\end{eqnarray}

\item
 The  $w$-weighted   boundary  map  is  a  homomorphism  of  graded  $R$-modules  
\begin{eqnarray*}
\partial_n(w):~~~  R_n(\mathcal{K})\longrightarrow    R_{n-1}(\mathcal{K})
\end{eqnarray*}
given  by  (cf.  \cite{daw, chengyuan1,chengyuan2, chengyuan3})
\begin{eqnarray}\label{eq-99v1}
\partial_n(w) \{v_0,v_1,\ldots, v_n\}=\sum_{i=0}^n  (-1)^i  w(v_i) \{v_0,\ldots, \widehat{v_i},\ldots, v_n\}.
\end{eqnarray}
It  is  direct   that   $\partial_{n}(w)\circ \partial_{n+1}(w)=0$.  
We  have   
  \begin{eqnarray}\label{eq-6.afv88bbb}
\partial_n(w) =\sum_{v\in V}  w(v)  \frac{\partial}{\partial v}.
\end{eqnarray}
The  $w$-weighted  homology  of  $\mathcal{K}$  is  
\begin{eqnarray}\label{eq-whm12}
  H_n(\mathcal{K},w;R)= {\rm Ker}(\partial_n(w))/{\rm Im}(\partial_{n+1}(w)).  
\end{eqnarray}
\item
The  $w$-weighted   coboundary  map  is  a  homomorphism  of  graded  $R$-modules  
\begin{eqnarray*}
\partial^n(w):~~~  R_n(\mathcal{L})\longrightarrow    R_{n+1}(\mathcal{L})
\end{eqnarray*}
given  by  
\begin{eqnarray}\label{eq-99v281}
\partial^n(w) \{v_0,v_1,\ldots, v_n\}=\sum_{v\in V}  {{\rm sgn}(v)}  w(v) \{v_0,v_1,\ldots,v_n,v\}.
\end{eqnarray}
It  is  direct   that   $\partial^{n+1}(w)\circ \partial^{n}(w)=0$.  
  We  have    
  \begin{eqnarray}\label{eq-6.afv88aaa}
\partial^n(w) =\sum_{v\in V}  w(v) dv.
\end{eqnarray}
The  $w$-weighted  cohomology  of  $\mathcal{L}$  is  
\begin{eqnarray}\label{eq-whm1290}
  H^n(\mathcal{L},w;R)= {\rm Ker}(\partial^n(w))/{\rm Im}(\partial^{n-1}(w)).  
\end{eqnarray}
\item
If  $w(v)\equiv  1$,  then   (\ref{eq-99v1}),  (\ref{eq-6.afv88bbb})  and  (\ref{eq-whm12})  
are  reduced  to   (\ref{eq-usual-boundary}),  (\ref{eq-6.afv1})  and   (\ref{eq-ih-a1})   respectively;  
and      (\ref{eq-99v281}),  (\ref{eq-6.afv88aaa})  and  (\ref{eq-whm1290})  
are  reduced  to   (\ref{usual-cob}),  (\ref{eq-6.afv2})  and   (\ref{eq-ih-a2})   respectively. 
\end{enumerate}
\end{example}

 \begin{example}\label{ex-8.778}
  Let   $V=\{v_0,v_1,v_2\}$.      
\begin{enumerate}[(1).]
 \item
 Take  the  simplicial  complexes 
 \begin{eqnarray*}
 \mathcal{K}_1&=& \{\{v_0,v_1\},  \{v_0,v_2\}, \{v_0\},  \{v_1\},  \{v_2\}\},  \\
 \mathcal{K}_2&=& \{\{v_0,v_1\},  \{v_1,v_2\}, \{v_0\},  \{v_1\},  \{v_2\}\}. 
 \end{eqnarray*}
 By  a  direct  calculation, 
 \begin{eqnarray*}
 &\dfrac{\partial}{\partial  v_0}\{v_0,v_1\}  =  \{v_1\},~~~~~~   \dfrac{\partial}{\partial  v_1}\{v_0,v_1\}  = - \{v_0\}, 
 ~~~~~~ \dfrac{\partial}{\partial  v_2}\{v_0,v_1\}  = 0,\\
 &\dfrac{\partial}{\partial  v_0}\{v_1,v_2\}  =  0,~~~~~~   \dfrac{\partial}{\partial  v_1}\{v_1,v_2\}  =   \{v_2\}, 
 ~~~~~~ \dfrac{\partial}{\partial  v_2}\{v_1,v_2\}  = -\{v_1\},\\
 &\dfrac{\partial}{\partial  v_0}\{v_0,v_2\}  =  \{v_2\},~~~~~~   \dfrac{\partial}{\partial  v_1}\{v_0,v_2\}  = 0, 
 ~~~~~~ \dfrac{\partial}{\partial  v_2}\{v_0,v_2\}  = -\{v_0\}. 
 \end{eqnarray*}
Thus 
 \begin{eqnarray*}
& H_0(\mathcal{K}_1, \dfrac{\partial}{\partial  v_0})=  R(\{v_0\}, \{v_1\}, \{v_2\})  /R(\{v_1\}, \{v_2\})=  R(\{v_0\})\cong  R,  \\
& H_0(\mathcal{K}_1, \dfrac{\partial}{\partial  v_1})=  H_0(\mathcal{K}_1, \dfrac{\partial}{\partial  v_2})= R(\{v_0\}, \{v_1\}, \{v_2\})  /R(\{v_0\} )=  R(\{v_1\},  \{v_2\})\cong  R^{\oplus  2},\\
& H_0(\mathcal{K}_2, \dfrac{\partial}{\partial  v_1})=  R(\{v_0\}, \{v_1\}, \{v_2\})  /R(\{v_0\}, \{v_2\})=  R(\{v_1\})\cong  R,\\
& H_0(\mathcal{K}_2, \dfrac{\partial}{\partial  v_0})=H_0(\mathcal{K}_2, \dfrac{\partial}{\partial  v_2})=  R(\{v_0\}, \{v_1\}, \{v_2\})  / R(\{v_1\} )=  R(\{v_0\},  \{v_2\} )\cong  R^{\oplus  2}  
 \end{eqnarray*}
 and 
 \begin{eqnarray*}
 H_1(\mathcal{K}_i,   \dfrac{\partial}{\partial  v_j})=0,  ~~~~~~  i=1,2,  ~~~ j=0,1,2.  
 \end{eqnarray*}
 Moreover,  
 \begin{eqnarray*}
 \mathcal{K}_1\cap\mathcal{K}_2= \{\{v_0,v_1\},    \{v_0\},  \{v_1\},  \{v_2\}\}.
 \end{eqnarray*}
  Thus  
   \begin{eqnarray*}
& H_0(\mathcal{K}_1\cap\mathcal{K}_2, \dfrac{\partial}{\partial  v_0})=  R(\{v_0\}, \{v_1\}, \{v_2\})  /R(\{v_1\})=  R(\{v_0\}, \{v_2\})\cong  R^{\oplus  2},  \\
& H_0(\mathcal{K}_1\cap\mathcal{K}_2, \dfrac{\partial}{\partial  v_1})=   R(\{v_0\}, \{v_1\}, \{v_2\})  /R(\{v_0\} )=  R(\{v_1\},  \{v_2\})\cong  R^{\oplus  2},\\
& H_0(\mathcal{K}_1\cap\mathcal{K}_2, \dfrac{\partial}{\partial  v_2})=   R(\{v_0\}, \{v_1\}, \{v_2\})  \cong  R^{\oplus  3}
 \end{eqnarray*}
 and 
  \begin{eqnarray*}
 H_1(\mathcal{K}_1\cap\mathcal{K}_2,   \dfrac{\partial}{\partial  v_j})=0,  ~~~~~~   j=0,1,2.  
 \end{eqnarray*}
 Furthermore,    
 \begin{eqnarray*}
 \mathcal{K}_2\cup\mathcal{K}_2= \{\{v_0,v_1\},  \{v_1,v_2\}, \{v_0,v_2\}, \{v_0\},  \{v_1\},  \{v_2\}\}. 
 \end{eqnarray*}
 Thus  
   \begin{eqnarray*}
& H_0(\mathcal{K}_1\cup\mathcal{K}_2, \dfrac{\partial}{\partial  v_0})=  R(\{v_0\}, \{v_1\}, \{v_2\})  /R(\{v_1\}, \{v_2\})=  R(\{v_0\} )\cong  R,  \\
& H_0(\mathcal{K}_1\cup\mathcal{K}_2, \dfrac{\partial}{\partial  v_1})=  R(\{v_0\}, \{v_1\}, \{v_2\})  /R(\{v_0\}, \{v_2\})=  R(\{v_1\} )\cong  R,  \\
& H_0(\mathcal{K}_1\cup\mathcal{K}_2, \dfrac{\partial}{\partial  v_2})=  R(\{v_0\}, \{v_1\}, \{v_2\})  /R(\{v_0\}, \{v_1\})=  R(\{v_2\} )\cong  R  
 \end{eqnarray*}
 and 
  \begin{eqnarray*}
 H_1(\mathcal{K}_1\cup\mathcal{K}_2,   \dfrac{\partial}{\partial  v_j})=0,  ~~~~~~   j=0,1,2.  
 \end{eqnarray*}
 Note that the  usual  simplicial  homology  of  $\mathcal{K}_1\cup\mathcal{K}_2$  is 
   \begin{eqnarray*}
 H_1(\mathcal{K}_1\cup\mathcal{K}_2,   \sum_{j=0}^2\dfrac{\partial}{\partial  v_j})=  R (\{v_0,v_1\}-\{v_0,v_2\} +\{v_1,v_2\})\cong   R.   
 \end{eqnarray*}
  \item
  Take  the  independence  hypergraphs
 \begin{eqnarray*}
 \mathcal{L}_1&=& \{\{v_0,v_1,v_2\},  \{v_0,v_2\}, \{v_0,v_1\},  \{v_0\}     \},  \\
 \mathcal{L}_2&=& \{\{v_0,v_1,v_2\},    \{v_0,v_2\},  \{v_1,v_2\},    \{v_2\}\}. 
 \end{eqnarray*}
By  a  direct  calculation, 
 \begin{eqnarray*}
 &d v_0 \{v_0 \}  =  0,~~~~~~   d v_0 \{v_2 \}  =   \{v_0,v_2\}, ~~~~~~d v_1 \{v_0 \}  =  -\{v_0,v_1\}, \\
& d v_1 \{v_2 \}  =   \{v_1,v_2\},~~~~~~ d v_2 \{v_0 \}  =  -\{v_0,v_2\},~~~~~~  dv_2  \{v_2\}  =0, \\
 & dv_0  \{v_0,v_2\} =  dv_2  \{v_0,v_2\}  =  dv_0 \{v_0,v_1\}  =  dv_1\{v_0,v_1\}  =  dv_1 \{v_1,v_2\}  = dv_2\{v_1,v_2\} =0, \\
 &  dv_0 \{v_1,v_2\}  = -  dv_1 \{v_0,v_2\}  =  dv_2\{v_0,v_1\} = \{v_0,v_1,v_2\}. 
 \end{eqnarray*}
 Thus 
  \begin{eqnarray*}
& H^0(\mathcal{L}_1,  d  v_0 )=   R(\{v_0\})\cong  R,  \\
& H^0(\mathcal{L}_1,  d  v_1 )= H^0(\mathcal{L}_1,  d  v_2 )=  0,  \\
& H^0(\mathcal{L}_2,  d  v_2 )=   R(\{v_2\})\cong  R,  \\
& H^0(\mathcal{L}_2,  d  v_0 )=H^0(\mathcal{L}_2,  d  v_1 )=0, 
\end{eqnarray*}  
  \begin{eqnarray*}
& H^1(\mathcal{L}_1,  d  v_0 )=   R(\{v_0,v_2\}, \{v_0,v_1\}) \cong  R^{\oplus  2},  \\
& H^1(\mathcal{L}_1,  d  v_1 )= R(  \{v_0,v_1\})/ R(  \{v_0,v_1\})  =0,  \\
& H^1(\mathcal{L}_1,  d  v_2 )= R(  \{v_0,v_2\})/ R(  \{v_0,v_2\})  =0,  \\
& H^1(\mathcal{L}_2,  d  v_2 )=R(\{v_0,v_2\}, \{v_1,v_2\}) \cong  R^{\oplus  2},\\
& H^1(\mathcal{L}_2,  d  v_0 )=R(  \{v_0,v_2\})/  R(\{v_0,v_2\})=0,  \\
& H^1(\mathcal{L}_2,  d  v_1 )=  R(  \{v_1,v_2\})/  R(\{v_1,v_2\})=0  
\end{eqnarray*}
and 
\begin{eqnarray*}
& H^2(\mathcal{L}_1,  d  v_0 )=   R(\{v_0,v_1,v_2\})  \cong  R,  \\
& H^2(\mathcal{L}_1,  d  v_1 )=  H^2(\mathcal{L}_1,  d  v_2 )= R(  \{v_0,v_1,v_2\})/ R(  \{v_0,v_1,v_2\})  =0,  \\
& H^2(\mathcal{L}_2,  d  v_2 )=R(\{v_0,v_1,v_2\})  \cong  R,\\
& H^2(\mathcal{L}_2,  d  v_0 )= H^2(\mathcal{L}_2,  d  v_1 )=  R(  \{v_0,v_1,v_2\})/ R(  \{v_0,v_1,v_2\})  =0. 
\end{eqnarray*}
Moreover,  
\begin{eqnarray*}
\mathcal{L}_1\cap\mathcal{L}_2 = \{\{v_0,v_1,v_2\},  \{v_0,v_2\}\}. 
\end{eqnarray*}
Thus 
\begin{eqnarray*}
&  H^0(\mathcal{L}_1\cap\mathcal{L}_2,  d  v_j)=   0, ~~~~~~  j=0,1,2,  \\
&  H^1(\mathcal{L}_1\cap\mathcal{L}_2,  d  v_0 )=     H^1(\mathcal{L}_1\cap\mathcal{L}_2,  d  v_2 )=R( \{v_0,v_2\})\cong   R, \\
&H^1(\mathcal{L}_1\cap\mathcal{L}_2,  d  v_1 )=0,\\
&H^2(\mathcal{L}_1\cap\mathcal{L}_2,  d  v_1 )=R(\{v_0,v_1,v_2\})/R(\{v_0,v_1,v_2\})=0,\\
&H^2(\mathcal{L}_1\cap\mathcal{L}_2,  d  v_0 )=H^2(\mathcal{L}_1\cap\mathcal{L}_2,  d  v_2 )= R(\{v_0,v_1,v_2\})\cong  R.  
\end{eqnarray*}
Furthermore,
\begin{eqnarray*}
\mathcal{L}_1\cup\mathcal{L}_2 = \{\{v_0,v_1,v_2\},  \{v_0,v_1\},  \{v_0,v_2\}, \{v_1,v_2\}, \{v_0\}, \{v_2\}\}. 
\end{eqnarray*}
Thus  
\begin{eqnarray*}
&  H^0(\mathcal{L}_1\cup\mathcal{L}_2,  d  v_0)=R(\{v_0\})\cong  R,\\
&  H^0(\mathcal{L}_1\cup\mathcal{L}_2,  d  v_1)=0,\\
&  H^0(\mathcal{L}_1\cup\mathcal{L}_2,  d  v_2)=R(\{v_2\})\cong  R,
\end{eqnarray*}
\begin{eqnarray*}
&  H^1(\mathcal{L}_1\cup\mathcal{L}_2,  d  v_0)=R(\{v_0,v_1\}, \{v_0,v_2\})/ R(\{v_0,v_2\})  \cong  R,\\
&  H^1(\mathcal{L}_1\cup\mathcal{L}_2,  d  v_1)=R(\{v_0,v_1\}, \{v_1,v_2\})/ R(\{v_0,v_1\}, \{v_1,v_2\})=0,\\
&  H^1(\mathcal{L}_1\cup\mathcal{L}_2,  d  v_2)=R(\{v_0,v_2\}, \{v_1,v_2\})/ R(\{v_0,v_2\})\cong  R\end{eqnarray*}
and 
\begin{eqnarray*}
&  H^2(\mathcal{L}_1\cup\mathcal{L}_2,  d  v_0)=   H^2(\mathcal{L}_1\cup\mathcal{L}_2,  d  v_1)=  H^2(\mathcal{L}_1\cup\mathcal{L}_2,  d  v_2)\\
&=R(\{v_0,v_1,v_2\})/ R(\{v_0,v_1,v_2\})=0.
\end{eqnarray*}
\item
 Take  the  hypergraphs  
 \begin{eqnarray*}
 \mathcal{H}_1 &=& \{\{v_0,v_1\}, \{v_0,v_2\}, \{v_0\}, \{v_1\}\},\\
 \mathcal{H}_2 &=&  \{\{v_0,v_1,v_2\}, \{v_0,v_1\}, \{v_2\}\}. 
 \end{eqnarray*}
 Then  
 \begin{eqnarray*}
 \Delta\mathcal{H}_1 & = & \{\{v_0,v_1\}, \{v_0,v_2\}, \{v_0\}, \{v_1\}, \{v_2\}\},\\
  \delta\mathcal{H}_1 & = &  \{\{v_0,v_1\},  \{v_0\}, \{v_1\}  \},\\
  \bar\Delta\mathcal{H}_1 & = & \{\{v_0,v_1,v_2\},\{v_0,v_1\}, \{v_0,v_2\}, \{v_1,v_2\}, \{v_0\}, \{v_1\} \},\\
    \bar\delta\mathcal{H}_1 & = & \emptyset   
 \end{eqnarray*}
 and  
  \begin{eqnarray*}
 \Delta\mathcal{H}_2 & = & \{\{v_0,v_1,v_2\}, \{v_0,v_1\}, \{v_0,v_2\},\{v_1,v_2\},  \{v_0\}, \{v_1\}, \{v_2\}\},\\
  \delta\mathcal{H}_2 & = &  \{  \{v_2\}  \},\\
  \bar\Delta\mathcal{H}_2 & = & \{\{v_0,v_1,v_2\},\{v_0,v_1\}, \{v_0, v_2\},  \{v_1,v_2\},  \{v_2\}  \},\\
    \bar\delta\mathcal{H}_2 & = & \{\{v_0,v_1,v_2\},\{v_0,v_1\} \}.   
 \end{eqnarray*}
 The  constrained  homology  of  $ \delta\mathcal{H}_i$,   $\Delta\mathcal{H}_i$,  $i=1,2$,  
 $\delta\mathcal{H}_1\cap\delta\mathcal{H}_2$  and   $\Delta\mathcal{H}_1\cap\Delta\mathcal{H}_2$   can  be  
 calculated  analogous  to  (1).   The  constrained  cohomology  of  $\bar \delta\mathcal{H}_i$,    $\bar\Delta\mathcal{H}_i$,  $i=1,2$,  $\bar\delta\mathcal{H}_1\cap\bar\delta\mathcal{H}_2$  and   $\bar\Delta\mathcal{H}_1\cap\bar\Delta\mathcal{H}_2$  can  be  
 calculated  analogous  to  (2).  
  \end{enumerate}
 \end{example}
 
    \begin{example}
 Let  $V=\mathbb{Z}$.  
 \begin{enumerate}[(1).]
 \item
 Take  the   hypergraph 
 \begin{eqnarray*}
 \mathcal{H}= \{\{v_i\}\mid  i\in\mathbb{Z}\}\cup   \{\{v_i,v_{i+1},v_{i+2}\}\mid  i\in \mathbb{Z}\}. 
 \end{eqnarray*}
 Then   the  lower-associated  simplicial  complex  is
 \begin{eqnarray*}
  \delta\mathcal{H} = \{\{v_i\}\mid  i\in\mathbb{Z}\},
  \end{eqnarray*}
  the  associated  simplicial  complex  is  
  \begin{eqnarray*} 
   \Delta\mathcal{H} = \{\{v_i\}\mid  i\in\mathbb{Z}\}\cup \{\{v_i,v_{i+1}\},\{v_i,v_{i+2}\}\mid  i\in\mathbb{Z}\}\cup   \{\{v_i,v_{i+1},v_{i+2}\}\mid  i\in \mathbb{Z}\}, 
   \end{eqnarray*}
   the  lower-associated  independence  hypergraph  is 
   \begin{eqnarray*}
\bar\delta\mathcal{H} = \emptyset, 
\end{eqnarray*}
and  the  associated  independence  hypergraph  is 
\begin{eqnarray*}
\bar\Delta\mathcal{H} =   \Delta[\mathbb{Z}] =\{ \sigma\subset \mathbb{Z} \mid  \sigma\neq\emptyset{\rm~and~} \sigma{\rm~is~finite}\}.  
 \end{eqnarray*}
 \item
 Take  the  hypergraph 
 \begin{eqnarray*}
 \mathcal{H} = \{ \sigma\subset \mathbb{Z} \mid   a\leq  |\sigma| \leq   b{\rm~and~} \sigma{\rm~is~finite}\} 
 \end{eqnarray*}
 where  $1\leq  a\leq  b\leq  +\infty$.  
 Then  the  lower-associated  simplicial  complex  is
 \begin{eqnarray*}
  \delta\mathcal{H} = \begin{cases}
 \emptyset,  &{\rm~if~} a>1,\\
 \mathcal{H},  &  {\rm~if~} a=1, 
 \end{cases} 
 \end{eqnarray*}
  the  associated  simplicial  complex  is  
  \begin{eqnarray*}
  \Delta\mathcal{H} = \{ \sigma\subset \mathbb{Z} \mid    |\sigma|  \leq  b,   \sigma\neq \emptyset  {\rm~and~} \sigma{\rm~is~finite}\},
  \end{eqnarray*} 
   the  lower-associated  independence  hypergraph  is 
  \begin{eqnarray*}
 \bar\delta\mathcal{H} =
\begin{cases}
 \emptyset,  &{\rm~if~}  b<+\infty,\\
  \{ \sigma\subset \mathbb{Z} \mid   a\leq  |\sigma|  {\rm~and~} \sigma{\rm~is~finite}\},  &  {\rm~if~} b=+\infty, 
 \end{cases}
 \end{eqnarray*} 
  and  the  associated  independence  hypergraph  is 
  \begin{eqnarray*}
 \bar\Delta\mathcal{H} =  \{ \sigma\subset \mathbb{Z} \mid   a\leq  |\sigma|  {\rm~and~} \sigma{\rm~is~finite}\}.  
 \end{eqnarray*}
 \end{enumerate}
 We  refer  to  Example~\ref{ex-8.778}~(1)  for   the  calculation   of  the  constrained  homology  of  $\delta\mathcal{H}$  as  well as   the  constrained  homology  of   $\Delta\mathcal{H}$  and   refer  to  Example~\ref{ex-8.778}~(2)   for   the  calculation    of   the     constrained  cohomology  of  $\bar\delta\mathcal{H}$  as  well as  the  constrained  cohomology  of    $\bar\Delta\mathcal{H}$.   
 \end{example}

\section{Mayer-Vietoris  Sequences  for  the  Constrained Persistent (co)Homology  of  Hypergraphs}
  \label{s88908}

Let  $\{\mathcal{H}_x\}_{x\in\mathbb{R}}$  be  a  filtration  of  hypergraphs     with  vertices  from   $V$.   Then  
 (i).   $\mathcal{H}_x$  is  a  hypergraph  with its  vertices from  $V$  for  any   $x\in  \mathbb{R}$      and  
  (ii).  there  is  a  canonical  inclusion  $\iota_x^y:   \mathcal{H}_x\longrightarrow \mathcal{H}_y$  of  hypergraphs  induced  by  the  identity  map  on  $V$   
   for  any  $-\infty<x\leq  y<+\infty$  such  that   $\iota_x^x= {\rm  id}$  and   $\iota_x^z  =  \iota_y^z \circ  \iota_x^y$  for  any   $-\infty<x\leq  y\leq  z<+\infty$.

   Recall that a persistent  $R$-module is a family  $V = \{V_x\}_{x\in  \mathbb{R}} $   of  $R$-modules, together with a family of homomorphisms  $h_x^y:  V_x  \longrightarrow  V_y $   for any   $-\infty<x\leq   y<+\infty$,  such  that     $h_x^x ={\rm  id}$  for  any  $x\in \mathbb{R}$  and  $ h_x^z=h_y^z \circ  h_x^y  $     for  any  $-\infty<x\leq  y\leq  z<+\infty$    (cf. \cite[Definition 2.1]{pd2}).   By  applying the  usual  homology  functor,  a   filtration  of   simplicial  complexes  $\{\mathcal{K}_x\}_{x\in \mathbb{R}}$  will  induce  a persistent $R$-module  
   $\{H_n(\mathcal{K}_x)\}_{x\in\mathbb{R}}$  for  any  $n\in \mathbb{N}$,  which  is  called  the  $n$-th  (or  the  $n$-dimensional)  persistent  homology    of  $\{\mathcal{K}_x\}_{x\in \mathbb{R}}$  (cf.  \cite{pd2,pmd,pd1}).

   Taking  the  associated  simplicial  complexes  and  lower-associated  simplicial  complexes
    for  the  filtration  of  hypergraphs  $\{\mathcal{H}_x\}_{x\in\mathbb{R}}$,  with  the  help  of  
     (\ref{eq-vmbg1}),   we   have   
   a  commutative  diagram  
   \begin{eqnarray*} 
  \xymatrix{
  \delta(\mathcal{H}_x) \ar[r]^-{\delta(\iota_x^y)} \ar[d]   & \delta(\mathcal{H}_y ) \ar[d] \\
  \mathcal{H}_x\ar[r]^-{\iota_x^y} \ar[d] &\mathcal{H}_y  \ar[d] \\
    \Delta(\mathcal{H}_x )\ar[r]^-{\Delta(\iota_x^y)}    & \Delta(\mathcal{H}_y  )
  }
  \end{eqnarray*}
for  any  $-\infty<x\leq  y<+\infty$.   Consequently,    both   $\{\delta(\mathcal{H}_x)\}_{x\in\mathbb{R}}$   equipped  with 
$\{\delta(\iota_x^y)\}_{-\infty<x\leq  y<+\infty}$  and      $\{\Delta(\mathcal{H}_x)\}_{x\in\mathbb{R}}$   equipped  with 
$\{\Delta(\iota_x^y)\}_{-\infty<x\leq  y<+\infty}$   
are  filtrations  of  simplicial  complexes.   
 Similarly,   taking  the  associated  independence  hypergraphs    and  lower-associated   independence  hypergraphs    for  the  filtration  of  hypergraphs  $\{\mathcal{H}_x\}_{x\in\mathbb{R}}$,  with  the  help  of  
     (\ref{eq-vmbg2}),   we   have   
   a  commutative  diagram  
   \begin{eqnarray*} 
  \xymatrix{
 \bar \delta(\mathcal{H}_x) \ar[r]^-{\bar\delta(\iota_x^y)} \ar[d]   & \bar\delta(\mathcal{H}_y ) \ar[d] \\
  \mathcal{H}_x\ar[r]^-{\iota_x^y} \ar[d] &\mathcal{H}_y  \ar[d] \\
   \bar \Delta(\mathcal{H}_x )\ar[r]^-{\bar\Delta(\iota_x^y)}    & \bar\Delta(\mathcal{H}_y  )
  }
  \end{eqnarray*}
for  any  $-\infty<x\leq  y<+\infty$.   Consequently,    both   $\{\bar\delta(\mathcal{H}_x)\}_{x\in\mathbb{R}}$   equipped  with 
$\{\bar\delta(\iota_x^y)\}_{-\infty<x\leq  y<+\infty}$  and      $\{\bar\Delta(\mathcal{H}_x)\}_{x\in\mathbb{R}}$   equipped  with 
$\{\bar\Delta(\iota_x^y)\}_{-\infty<x\leq  y<+\infty}$   
are  filtrations  of  independence  hypergraphs.

  Let     $\{\mathcal{H}_x\}_{x\in\mathbb{R}}$  and  $\{\mathcal{H}'_x\}_{x\in\mathbb{R}}$  
          be  two  filtrations  of  hypergraphs  with  their  vertices  from  $V$.  
            For  any  $-\infty<x\leq  y<+\infty$,  the  canonical  inclusions  
            $\iota_x^y: \mathcal{H}_x\longrightarrow  \mathcal{H}_y$  and 
              ${\iota'}_x^y: \mathcal{H}'_x\longrightarrow  \mathcal{H}'_y$  
              induce   
                a  morphism  of  the Mayer-Vietoris 
            sequences   of the constrained  homology  (cf.  \cite[Section~5]{mv})
            \begin{eqnarray}\label{eq-1oiu9801}
            (\delta(\iota_x^y), \delta({\iota'}_x^y))_*: ~~~  {\bf  MV}_*(\delta(\mathcal{H}_x),\delta(\mathcal{H}'_x),  \alpha,m)\longrightarrow {\bf  MV}_*(\delta(\mathcal{H}_y),\delta(\mathcal{H}'_y),\alpha,m)  
            \end{eqnarray}
as  well as  
              a  morphism  of  the Mayer-Vietoris 
            sequences   of the constrained  homology   (cf.  \cite[Section~5]{mv}) 
            \begin{eqnarray}\label{eq-1oiu9802}
            (\Delta(\iota_x^y), \Delta({\iota'}_x^y))_*: ~~~  {\bf  MV}_*(\Delta(\mathcal{H}_x),\Delta(\mathcal{H}'_x),  \alpha,m)\longrightarrow {\bf  MV}_*(\Delta(\mathcal{H}_y),\Delta(\mathcal{H}'_y),\alpha,m). 
            \end{eqnarray}
            The  following  diagram   commutes  
            \begin{eqnarray*} 
            \xymatrix{
            {\bf  MV}_*(\delta(\mathcal{H}_x),\delta(\mathcal{H}'_x),  \alpha,m)\ar[rr]^-{(\delta(\iota_x^y), \delta({\iota'}_x^y))_*}\ar[d]
            && {\bf  MV}_*(\delta(\mathcal{H}_y),\delta(\mathcal{H}'_y),\alpha,m)\ar[d]\\
              {\bf  MV}_*(\Delta(\mathcal{H}_x),\Delta(\mathcal{H}'_x),  \alpha,m)\ar[rr]^-{(\Delta(\iota_x^y), \Delta({\iota'}_x^y))_*}
            && {\bf  MV}_*(\Delta(\mathcal{H}_y),\Delta(\mathcal{H}'_y),\alpha,m)
            }
            \end{eqnarray*}
            where  the  vertical  maps  are  morphisms  of  Mayer-Vietoris  sequences   induced  by  the  canonical  inclusions  of  $\delta(\mathcal{H}_x)$  into  
            $\Delta(\mathcal{H}_x)$  and  the  canonical  inclusions  of  $\delta(\mathcal{H}'_x)$  into  
            $\Delta(\mathcal{H}'_x)$,  $x\in \mathbb{R}$.  
            Similarly,  
            for  any  $-\infty<x\leq  y<+\infty$,  the  canonical  inclusions  
            $\iota_x^y: \mathcal{H}_x\longrightarrow  \mathcal{H}_y$  and 
              $\iota'^y_x: \mathcal{H}'_x\longrightarrow  \mathcal{H}'_y$  
              induce   
                a  morphism  of  the Mayer-Vietoris 
            sequences   of the constrained  cohomology  
            \begin{eqnarray}\label{eq-1oiu9805}
            ({\bar\delta}(\iota_x^y), {\bar\delta}(\iota'^y_x))_*: ~~~  {\bf  MV}^*({\bar\delta}(\mathcal{H}_x),{\bar\delta}(\mathcal{H}'_x),  \omega,m)\longrightarrow {\bf  MV}^*({\bar\delta}(\mathcal{H}_y),{\bar\delta}(\mathcal{H}'_y),\omega,m)  
            \end{eqnarray}
as  well as  
              a  morphism  of  the Mayer-Vietoris 
            sequences   of the constrained  cohomology  
            \begin{eqnarray}\label{eq-1oiu9806}
            ({\bar\Delta}(\iota_x^y), {\bar\Delta}(\iota'^y_x))_*: ~~~  {\bf  MV}^*({\bar\Delta}(\mathcal{H}_x),{\bar\Delta}(\mathcal{H}'_x),  \omega,m)\longrightarrow {\bf  MV}^*({\bar\Delta}(\mathcal{H}_y),{\bar\Delta}(\mathcal{H}'_y),\omega,m). 
            \end{eqnarray}
            The  following  diagram   commutes  
            \begin{eqnarray*} 
            \xymatrix{
            {\bf  MV}^*(\bar\delta(\mathcal{H}_x),\bar\delta(\mathcal{H}'_x),  \omega,m)\ar[rr]^-{(\bar\delta(\iota_x^y), \bar\delta({\iota'}_x^y))_*}\ar[d]
            && {\bf  MV}^*(\bar\delta(\mathcal{H}_y),\bar\delta(\mathcal{H}'_y),\alpha,m)\ar[d]\\
              {\bf  MV}^*(\bar\Delta(\mathcal{H}_x),\bar\Delta(\mathcal{H}'_x),  \alpha,m)\ar[rr]^-{(\bar\Delta(\iota_x^y), \bar\Delta({\iota'}_x^y))_*}
            && {\bf  MV}^*(\bar\Delta(\mathcal{H}_y),\bar\Delta(\mathcal{H}'_y),\alpha,m)
            }
            \end{eqnarray*}
            where  the  vertical  maps  are  morphisms  of  Mayer-Vietoris  sequences   induced  by  the  canonical  inclusions  of  $\bar\delta(\mathcal{H}_x)$  into  
            $\bar\Delta(\mathcal{H}_x)$  and  the  canonical  inclusions  of  $\bar\delta(\mathcal{H}'_x)$  into  
            $\bar\Delta(\mathcal{H}'_x)$,  $x\in \mathbb{R}$.

            We  call the  families  $\{{\bf  MV}_*(\delta(\mathcal{H}_x),\delta(\mathcal{H}'_x),\alpha,m)\}_{x\in\mathbb{R}}$   and  $\{{\bf  MV}_*(\Delta(\mathcal{H}_x),\Delta(\mathcal{H}'_x),\alpha,m)\}_{x\in\mathbb{R}}$ together with the  families  of  morphisms     given by  (\ref{eq-1oiu9801})  and  (\ref{eq-1oiu9802})   the  
            {\it    persistent  Mayer-Vietoris  sequence}  of  the  constrained  persistent homology  for the filtrations 
          of  hypergraphs   $\{\mathcal{H}_x\}_{x\in\mathbb{R}}$ and  $\{\mathcal{H}'_x\}_{x\in\mathbb{R}}$.   We    denote this persistent  Mayer-Vietoris  sequence  as  
            \begin{eqnarray}\label{eofm1}
            {\bf   PMV}_*(\mathcal{H}_x,\mathcal{H}'_x,  \alpha,m\mid  x\in\mathbb{R}).  
            \end{eqnarray}            
           We  call the  families  $\{{\bf  MV}^*(\bar\delta(\mathcal{H}_x),\bar\delta(\mathcal{H}'_x),\omega,m)\}_{x\in\mathbb{R}}$   and  $\{{\bf  MV}^*(\bar\Delta(\mathcal{H}_x),\bar\Delta(\mathcal{H}'_x),\omega,m)\}_{x\in\mathbb{R}}$ together with the  families  of  morphisms     given by  (\ref{eq-1oiu9805})  and  (\ref{eq-1oiu9806})   the  
            {\it    persistent  Mayer-Vietoris  sequence}  of  the  constrained  persistent  cohomology  for the filtrations 
           of  hypergraphs   $\{\mathcal{H}_x\}_{x\in\mathbb{R}}$  and  $\{\mathcal{H}'_x\}_{x\in\mathbb{R}}$.   We    denote this persistent  Mayer-Vietoris  sequence  as  
            \begin{eqnarray}\label{eofm2}
            {\bf   PMV}^*(\mathcal{H}_x,\mathcal{H}'_x,  \omega,m\mid  x\in\mathbb{R}).  
            \end{eqnarray} 
  We  prove  the  following  theorem  for    the   Mayer-Vietoris  sequences (\ref{eofm1})  and   (\ref{eofm2}).

  \begin{theorem}[Main  Result  III]
  \label{th-091286}
   For  any   filtrations   of  hypergraphs    $\{\mathcal{H}_x\}_{x\in\mathbb{R}}$  and  $\{\mathcal{H}'_x\}_{x\in\mathbb{R}}$  
               with  their  vertices  from  $V$,   
    \begin{enumerate}[(1).]
    \item
    we  have  an  induced  morphism  
\begin{eqnarray}\label{eq-5.2vla1ppp}
\beta_*: ~~~             {\bf   PMV}_*(\mathcal{H}_x,\mathcal{H}'_x,  \alpha,m\mid  x\in\mathbb{R})   
\longrightarrow              {\bf   PMV}_*(\mathcal{H}_x,\mathcal{H}'_x,  \alpha,m-2s\mid  x\in\mathbb{R}).  
\end{eqnarray}
 Moreover,   $(\beta_1\wedge\beta_2)_*= (\beta_1)_* (\beta_2)_*= (\beta_2)_* (\beta_1)_*$,     i.e.  the  diagram  commutes 
 \begin{eqnarray}\label{eq-5.2vlb1ppp}
 \xymatrix{
            {\bf   PMV}_*(\mathcal{H}_x,\mathcal{H}'_x,  \alpha,m\mid  x\in\mathbb{R})   
 \ar[r]^-{(\beta_1)_*} \ar[rdd]^-{(\beta_1\wedge \beta_2)_*}
 \ar[dd]_-{(\beta_2)_*}  & {\bf   PMV}_*(\mathcal{H}_x,\mathcal{H}'_x,  \alpha,m-2s_1\mid  x\in\mathbb{R})\ar[dd]^-{(\beta_2)_*}\\
 \\
{\bf   PMV}_*(\mathcal{H}_x,\mathcal{H}'_x,  \alpha,m-2s_2\mid  x\in\mathbb{R})\ar[r]^-{(\beta_1)_*}& {\bf   PMV}_*(\mathcal{H}_x,\mathcal{H}'_x,  \alpha,m-2s_1-2s_2\mid  x\in\mathbb{R});   
 }
 \end{eqnarray}

    \item
      we  have  an  induced  morphism  
\begin{eqnarray}\label{eq-5.2vla2ppp}
\mu_*: ~~~ {\bf   PMV}^*(\mathcal{H}_x,\mathcal{H}'_x,  \omega,m\mid  x\in\mathbb{R})\longrightarrow  {\bf   PMV}^*(\mathcal{H}_x,\mathcal{H}'_x,  \omega,m+2s\mid  x\in\mathbb{R}). 
\end{eqnarray}
Moreover,   $(\mu_1\wedge\mu_2)_*= (\mu_1)_* (\mu_2)_*= (\mu_2)_* (\mu_1)_*$,     i.e.  the  diagram  commutes 
 \begin{eqnarray}\label{eq-5.2vlb2ppp}
 \xymatrix{
{\bf   PMV}^*(\mathcal{H}_x,\mathcal{H}'_x,  \omega,m\mid  x\in\mathbb{R})\ar[r]^-{(\mu_1)_*} \ar[rdd]^-{(\mu_1\wedge \mu_2)_*}
 \ar[dd]_-{(\mu_2)_*}  &{\bf   PMV}^*(\mathcal{H}_x,\mathcal{H}'_x,  \omega,m+2s_1\mid  x\in\mathbb{R})\ar[dd]^-{(\mu_2)_*}\\
 \\
{\bf   PMV}^*(\mathcal{H}_x,\mathcal{H}'_x,  \omega,m+2s_2\mid  x\in\mathbb{R})\ar[r]^-{(\mu_1)_*}&{\bf   PMV}^*(\mathcal{H}_x,\mathcal{H}'_x,  \omega,m+2s_1+2s_2\mid  x\in\mathbb{R}).  
 }
 \end{eqnarray}
\end{enumerate}
  \end{theorem}
  
  \begin{proof}
  (1).    By  Theorem~\ref{pr-5.829a}~(1),  both  (\ref{eq-5.2vla1})  and   (\ref{eq-5.2vlb1})  are 
    functorial  with  respect to  morphisms  of  hypergraphs   induced  by  bijective  maps  between  the  vertices.  
    Since  the  morphisms  in  the   filtration  are   induced  by  the  identity  map  on the  vertices,  
    we  can  take     the  persistence  of   (\ref{eq-5.2vla1})  and   (\ref{eq-5.2vlb1}).  
    We  obtain   (\ref{eq-5.2vla1ppp})  and   (\ref{eq-5.2vlb1ppp})   respectively.

  (2).  Analogous to  (1),  we  take    the  persistence  of   (\ref{eq-5.2vla2})  and   (\ref{eq-5.2vlb2}).  
    We  obtain   (\ref{eq-5.2vla2ppp})  and   (\ref{eq-5.2vlb2ppp})   respectively. 
  \end{proof}

 \section{Applications  to        Persistent  Homology  for  Large   Networks}\label{s5}
 
 We  discuss     some  applications   of  the  constrained  (co)homology  of  hypergraphs    to  the      persistent homology   of   large networks.      Throughout  this  section,  we  use  a  hypergraph  $\mathcal{H}$  to  represent  a  network.  
 
 \subsection{Hypergraph  representations   for   sparse  networks and  dense  networks}

Let  $V$  be  a  finite  set.  
Let  $\mathcal{H}$  be  a  hypergraph  with  its  vertices from  $V$.
Let  $n\in \mathbb{N}$.    
The  density  of  the  $n$-hyperedges  in  $\mathcal{H}$  can  be  measured  by  the  ratio 
\begin{eqnarray*}
m_n(\mathcal{H})= \frac{|\mathcal{H}_n|}{ |{ V \choose   {n+1}}| } 
\end{eqnarray*}
where  $\mathcal{H}_n$  is  the  sub-hypergraph  of  $\mathcal{H}$  consisting  of  the  $n$-hyperedges  
in  $\mathcal{H}$  and  $V \choose   {n+1}$   is  the  hypergraph  consisting   of  all the  possible  $n$-hyperedges 
with  vertices  from  $V$.  The  density  of  $\mathcal{H}$   can  be  measured  by  the  ratio 
\begin{eqnarray*}
m(\mathcal{H})= \frac{|\mathcal{H}|}{ |\Delta[V]| }.   
\end{eqnarray*}
Note  that  
\begin{eqnarray*}
|\mathcal{H}|= \sum_{n\geq  0}  |\mathcal{H}_n|,  ~~~~~~ |\Delta[V]|  =  \sum_{n\geq 0}   \Big|{ V \choose   {n+1}}\Big|.   
\end{eqnarray*}
We  say  that  $\mathcal{H}$  is  {\it  sparse}  if  $m(\mathcal{H})<<1$  and  say  that  
$\mathcal{H}$  is  {\it   dense}  if  $1- m(\mathcal{H}) <<1$.  
Similarly,    we  say  that  $\mathcal{H}$  is  {\it   $n$-sparse}  if  $m_n(\mathcal{H})<<1$  and  say  that  
$\mathcal{H}$  is  {\it   $n$-dense}  if  $1- m_n(\mathcal{H}) <<1$.

Let  $\mathcal{K}$  be  a  simplicial complex  with  vertices  from  $V$.  
Let  $\sigma\in \Delta[V]\setminus  \mathcal{K}$.  
We  say that  $\sigma$  is  an  {\it  external  face}  of  $\mathcal{K}$  if  
$\tau\in \mathcal{K}$  for any proper  subset  
$\tau$  of  $\sigma$.  
Let  $E(\mathcal{K})$  be  the set  of  all the  external faces of $\mathcal{K}$.   
Let  $\mathcal{L}$  be  an  independence  hypergraph   with  vertices  from  $V$.  
Let  $\sigma\in \Delta[V]\setminus  \mathcal{L}$.  
We  say that  $\sigma$  is  a   {\it  co-external  face}  of  $\mathcal{L}$   if  
$\tau\in \mathcal{L}$  for any proper  superset  
$\tau$  of  $\sigma$,  where $\tau$  is a subset of  $V$.  
Let  $\bar  E(\mathcal{L})$  be  the set  of  all the  co-external faces of   $\mathcal{L}$.   
   Let    $p:  \Delta[V]\longrightarrow  [0,1]$. 
 Consider  the  random  hypergraph  whose  probability   is  given by 
  \begin{eqnarray*} 
  \bar{\rm P}_{p}(\mathcal{H})=\prod_{\sigma\in\mathcal{H}}  p(\sigma)\prod_{\sigma\notin   \mathcal{H}}\big(1-p(\sigma)\big) 
  \end{eqnarray*}
  where  $\mathcal{H}$  runs  over  all  hypergraphs  with  vertices  from  $V$,  
    the       random  simplicial  complex  whose  probability   is  given by 
  \begin{eqnarray*} 
  {\rm P}_{p}(\mathcal{K})=\prod_{\sigma\in\mathcal{K}}  p(\sigma)\prod_{\sigma\in  E(\mathcal{K})}\big(1-p(\sigma)\big) 
  \end{eqnarray*}
    where  $\mathcal{K}$  runs  over  all  simplicial  complexes  with  vertices  from  $V$,   
 and  the      random  independence  hypergraph  whose  probability   is  given by 
    \begin{eqnarray*} 
{\rm Q}_{p}(\mathcal{L})=\prod_{\sigma\in\mathcal{L}}  p(\sigma)\prod_{\sigma\in  \bar  E(\mathcal{L})}\big(1-p(\sigma)\big)  
  \end{eqnarray*}
    where  $\mathcal{L}$  runs  over  all  independence  hypergraphs  with  vertices  from  $V$.   
    Let  $\mathcal{H}\sim  \bar{\rm P}_{p}$   be  a  randomly  generated  hypergraph.   Then 
 we  have  (cf.  \cite[Section~3]{f2022},   \cite[Theorem~1.5 (2)]{jktr2},  \cite[Theorem~1.1]{grh1}) 
 \begin{enumerate}[(1).]      
\item
the  lower-associated  simplicial  complex   of   $\mathcal{H}$   is   a  randomly generated  
simplicial  complex   $\delta\mathcal{H} \sim  {\rm  P}_{p}$, 
\item
the  lower-associated  independence  hypergraph   of   $\mathcal{H}$     is   a   randomly generated  
 independence  hypergraph   $ \bar\delta\mathcal{H}\sim  {\rm  Q}_{p}$,   
\item
the  complement   of   $\mathcal{H}$  is  a randomly  generated  hypergraph  $\gamma\mathcal{H}\sim  \bar{\rm  P}_{1-p}$,  
\item
the  complement  of  the  associated  simplicial  complex  of   $\mathcal{H}$   is  a  randomly  generated  
independence  hypergraph  $\gamma \Delta\mathcal{H} \sim  {\rm  Q}_{1-p}$,  
\item
the  complement  of  the  associated  independence  hypergraph   of   $\mathcal{H}$   is a randomly generated  
simplicial  complex    $\gamma \bar\Delta\mathcal{H} \sim  {\rm   P}_{1-p}$.  
\end{enumerate}

  We   expect  to   use  simplicial  complexes  as  models  to  represent   sparse  networks  and  use  independence  hypergraphs  as  models  
to  represent   dense  networks.  For example,  if  the  networks  are  randomly  generated  by  $\bar{\rm  P}_p$,  
then  we  expect the  followings:

\begin{enumerate}[(1).]
\item
Suppose  $\frac{p(\sigma)}{|\sigma|}<<  1$  for  $\sigma\in \Delta[V]$  such that      $\mathcal{H}$  is  sparse   and   the  $n$-sparsity    of  $\mathcal{H}$   is  increasing  as  $n$  grows.  
   We  use  the  lower-associated   simplicial  complex   $\delta\mathcal{H} \sim  {\rm  P}_{p}$ 
   as  well as   the  associated  simplicial  complex     $\Delta\mathcal{H}$  satisfying  $\gamma \Delta\mathcal{H} \sim  {\rm  Q}_{1-p}$   
   to   represent   $\mathcal{H}$;     
   
\item
Let   $1-\frac{p(\sigma)}{|\sigma|}<<  1$  for  $\sigma\in \Delta[V]$  such  that  $\mathcal{H}$  is  dense    and   the  $n$-density    of  $\mathcal{H}$   is  increasing  as  $n$  grows.   
We  use  
the   associated   independence  hypergraph   $\bar\Delta\mathcal{H}$  satisfying   $\gamma \bar\Delta\mathcal{H} \sim  {\rm   P}_{1-p}$    as   well  as   the  lower-associated  independence  hypergraph   $ \bar\delta\mathcal{H}\sim  {\rm  Q}_{p}$   to   
 represent    $\mathcal{H}$.  
 \end{enumerate}

   \subsection{ Localizations of  persistent homology}

   Let  $V_1,  V_2,\ldots,  V_k$  be  a  partition  of  $V$,  i.e.  $V$  is  the  disjoint  union  of  $V_1,  V_2,\ldots,  V_k$.   
   Let  $\mathcal{K}$   be  a  simplicial  complex  with  its  vertices  from  $V$.  
    For  each  $i=1,2,\ldots, k$,   we  define  the  {\it  localized  boundary  map}  with  respect  to  $V_i$  as  a  homomorphism  of  graded  $R$-modules   
      \begin{eqnarray*}
   \partial_n(w,  V_i):~~~   R_n(\mathcal{K})\longrightarrow   R_{n-1}(\mathcal{K})
   \end{eqnarray*}
   given  by  
  \begin{eqnarray*}
    \partial_n(w,  V_i)=\sum_{v\in V_i}w(v)\frac{\partial}{\partial v}.
  \end{eqnarray*}
  We  have 
    \begin{eqnarray}\label{eq-6.afv897eee}
\partial_n(w) =\sum_{i=1}^k  \partial_n(w,  V_i). 
\end{eqnarray}
Take  the  homology  group 
\begin{eqnarray}\label{eq-whm1bttd}
  H_n(\mathcal{K},w, V_i;R)= {\rm Ker}(\partial_n(w,V_i))/{\rm Im}(\partial_{n+1}(w,V_i)).  
\end{eqnarray}
We  call  (\ref{eq-6.afv897eee})   the  {\it  localization}  of  the  boundary  map $\partial_n(w)$  with  respect to  $V_i$   and  call   (\ref{eq-whm1bttd})  the  {\it  localized  homology}  of  $\mathcal{K}$  with  respect  to  $w$ and  $V_i$. 
It  is  clear  that  (\ref{eq-whm1bttd})  is  the   constrained  homology  with  respect to  $\partial_n(w,V_i)\in  {\rm  Ext}_1(V)$.    
Let  $\mathcal{L}$  be  an  independence   hypergraph  with  its  vertices  from  $V$.    For  each  $i=1,2,\ldots, k$,   we  define  the  {\it  localized  coboundary  map}  with  respect  to  $V_i$  as  a  homomorphism  of  graded  $R$-modules   
      \begin{eqnarray*}
   \partial^n(w,  V_i):~~~   R_n(\mathcal{L})\longrightarrow   R_{n+1}(\mathcal{L})
   \end{eqnarray*}
   given  by  
  \begin{eqnarray*}
    \partial^n(w,  V_i)=\sum_{v\in V_i}w(v)  dv.
  \end{eqnarray*}
  We  have 
    \begin{eqnarray}\label{eq-6.afv897fff}
\partial^n(w) =\sum_{i=1}^k  \partial^n(w,  V_i). 
\end{eqnarray}
Take  the  cohomology  group 
\begin{eqnarray}\label{eq-whm1affzq}
  H^n(\mathcal{L},w, V_i;R)= {\rm Ker}(\partial^n(w,V_i))/{\rm Im}(\partial^{n-1}(w,V_i)).  
\end{eqnarray}
We  call  (\ref{eq-6.afv897fff})   the  {\it  localization}  of  the  coboundary  map $\partial^n(w)$  with  respect to  $V_i$     and   call  (\ref{eq-whm1affzq})  the  {\it  localized  cohomology}  of  $\mathcal{L}$  with  respect  to  $w$ and  $V_i$. 
It  is  clear  that  (\ref{eq-whm1affzq})  is  the   constrained   cohomology  with  respect to  $\partial^n(w,V_i)\in  {\rm  Ext}^1(V)$.    
 Represented  as  matrices,  the  size  of  $\partial_n(w,  V_i)$ is   smaller  than the  size   of  $\partial_n(w)$   and  the  size  of  $\partial^n(w,  V_i)$ is   smaller  than the  size   of  $\partial^n(w)$.  
 In  particular,   take   $w(v)\equiv  1$.   The  complexities  for  the  computation of    the  persistent  (co)homology  (cf.  \cite{2005})  would  be  reduced  
 if  we  substitute  the  usual  (co)boundary  maps    with  the   localized  (co)boundnary  maps.

\subsection{Perisistent localized  (co)homology  for  hypergraphs}

 By the  functoriality  of   the  constrained  homology  of  simplicial  complexes,    the  localized  homology  groups 
 \begin{eqnarray*}
  H_n( \delta(\mathcal{H}_x),  w,   V_i;  R),  ~~~~~~  x\in  \mathbb{R} 
 \end{eqnarray*}
 together  with the  induced  homomorphisms  
 \begin{eqnarray*}
(\delta(\iota_x^y))_*:~~~  H_n( \delta(\mathcal{H}_x),  w,   V_i;  R)\longrightarrow    H_n( \delta(\mathcal{H}_y),  w,   V_i;  R),  ~~~~~~  -\infty<x\leq  y<+\infty, 
 \end{eqnarray*}
   gives  a   persistent  $R$-module.  We  call   this  persistent  $R$-module  the       {\it   localized  persistent  homology}   of  the  induced  filtration  $\{\delta(\mathcal{H}_x)\}_{x\in\mathbb{R}}$  of   the  lower-associated  simplicial  complexes  and    denote  it  as   
   \begin{eqnarray}\label{eq-8.bvcx81}
   {\bf  PH}_n( \delta(\mathcal{H}_x),  w,   V_i;  R\mid  x\in \mathbb{R}).
   \end{eqnarray}
   The  localized  homology  groups  
 \begin{eqnarray*}
  H_n( \Delta(\mathcal{H}_x),  w,   V_i;  R),  ~~~~~~  x\in  \mathbb{R} 
 \end{eqnarray*}
 together  with the  induced  homomorphisms  
 \begin{eqnarray*}
(\Delta(\iota_x^y))_*:~~~  H_n( \Delta(\mathcal{H}_x),  w,   V_i;  R)\longrightarrow    H_n( \Delta(\mathcal{H}_y),  w,   V_i;  R),  ~~~~~~  -\infty<x\leq  y<+\infty, 
 \end{eqnarray*}
   gives  a  persistent  $R$-module.   We  call  this  persistent   $R$-module  the  {\it    localized  persistent  homology}   of  the  induced  filtration   $\{\Delta(\mathcal{H}_x)\}_{x\in\mathbb{R}}$  of   the  associated  simplicial  complexes  and   denote  it  as    
   \begin{eqnarray}\label{eq-8.bvcx82}
   {\bf  PH}_n( \Delta(\mathcal{H}_x),  w,   V_i;  R\mid  x\in \mathbb{R}). 
   \end{eqnarray}
     For  $\{\mathcal{H}_x\}_{x\in \mathbb{R}}$  representing  a  filtration  of    sparse  networks,   
   we  expect to  use  (\ref{eq-8.bvcx81})   and      (\ref{eq-8.bvcx82})  for  the  computations of   the $n$-th  persistent homology,  $n\in \mathbb{N}$.  
   Note that      
  (\ref{eq-8.bvcx81})   and      (\ref{eq-8.bvcx82})    are   special     constrained  persistent  homology  with  $\alpha\in  {\rm  Ext}_1(V_i)$  and   $m=0$.

  By the  functoriality  of   the  constrained  cohomology  of  independence  hypergraphs,    the  localized  cohomology  groups 
 \begin{eqnarray*}
  H^n( \bar\delta(\mathcal{H}_x),  w,   V_i;  R),  ~~~~~~  x\in  \mathbb{R} 
 \end{eqnarray*}
 together  with the  induced  homomorphisms  
 \begin{eqnarray*}
(\bar\delta(\iota_x^y))_*:~~~  H^n( \bar\delta(\mathcal{H}_x),  w,   V_i;  R)\longrightarrow    H^n( \bar\delta(\mathcal{H}_y),  w,   V_i;  R),  ~~~~~~  -\infty<x\leq  y<+\infty, 
 \end{eqnarray*}
   gives  a  persistent  $R$-module.  We  call  this  persistent   $R$-module  the  {\it  localized  persistent  cohomology}   of  the  induced  filtration   $\{\bar\delta(\mathcal{H}_x)\}_{x\in\mathbb{R}}$   of   the  lower-associated   independence  hypergraphs  and   denote  it  as    
\begin{eqnarray}\label{eq-0vsvjr1}
{\bf  PH}^n( \bar\delta(\mathcal{H}_x),  w,   V_i;  R\mid  x\in \mathbb{R}).    
\end{eqnarray}
 The  localized  cohomology  groups  
 \begin{eqnarray*}
  H^n( \bar\Delta(\mathcal{H}_x),  w,   V_i;  R),  ~~~~~~  x\in  \mathbb{R} 
 \end{eqnarray*}
 together  with the  induced  homomorphisms  
 \begin{eqnarray*}
(\bar\Delta(\iota_x^y))_*:~~~  H^n( \bar\Delta(\mathcal{H}_x),  w,   V_i;  R)\longrightarrow    H^n( \bar\Delta(\mathcal{H}_y),  w,   V_i;  R),  ~~~~~~  -\infty<x\leq  y<+\infty, 
 \end{eqnarray*}
   gives  a  persistent  $R$-module.   We  call  this  persistent   $R$-module  the  {\it   localized  persistent  cohomology}   of  the  induced  filtration   $\{\bar\Delta(\mathcal{H}_x)\}_{x\in\mathbb{R}}$   of   the    associated   independence  hypergraphs  and   denote  it  as   
   \begin{eqnarray}\label{eq-0vsvjr2}
   {\bf  PH}^n( \bar\Delta(\mathcal{H}_x),  w,   V_i;  R\mid  x\in \mathbb{R}).
   \end{eqnarray}
For  $\{\mathcal{H}_x\}_{x\in \mathbb{R}}$  representing  a  filtration  of    dense  networks,   
     we  expect to  use  (\ref{eq-0vsvjr1})   and      (\ref{eq-0vsvjr2})    
    for  the  computations of   the $n$-th  persistent  cohomology,  $n\in \mathbb{N}$.    
   Note  that   (\ref{eq-0vsvjr1})   and      (\ref{eq-0vsvjr2})      are   special     constrained  persistent  cohomology  with  $\omega\in  {\rm  Ext}^1(V_i)$  and   $m=0$.

  \section*{Acknowledgement} {The  author would like to express his  deep
gratitude to the referee for the careful reading of the manuscript.}

 \smallskip

 \smallskip
 {\small
Shiquan Ren

Address:   School of Mathematics and Statistics,  Henan University,  Kaifeng  475004,  China.

E-mail:  renshiquan@henu.edu.cn}

  \end{document}